\documentclass[a4paper, 11pt]{amsart}
\usepackage[utf8]{inputenc}
\usepackage[english]{babel}
\usepackage[autostyle]{csquotes}
\usepackage[style=numeric, backend=biber]{biblatex}
\usepackage{vmargin}
\setpapersize{A4}
\setmarginsrb{22mm}{15mm}{22mm}{15mm}{0mm}{10mm}{0mm}{10mm}
\usepackage[centertags]{amsmath}
\usepackage{amsthm}
\usepackage[bookmarks=false]{hyperref}
\usepackage{hyperref}
\usepackage[pass]{geometry}
\usepackage{amssymb}
\usepackage{enumitem}
\usepackage{mathtools}
\usepackage{mathrsfs}

\DeclarePairedDelimiterX{\scal}[2]{\langle}{\rangle}{#1, #2}
\DeclarePairedDelimiterX{\norm}[1]{\lVert}{\rVert}{#1}
\DeclarePairedDelimiterX{\normi}[1]{\lVert}{\rVert_\infty}{#1}
\DeclareMathOperator{\dive}{div}

\DeclareMathOperator{\supp}{supp}

\DeclareMathOperator{\loc}{loc}
\DeclareMathOperator{\tang}{Tan}
\DeclareMathOperator{\reg}{Dens}
\DeclareMathOperator{\sing}{Sing}
\DeclareMathOperator{\leb}{Leb}
\newcommand{\trasl}{L}
\newcommand{\dil}[1]{\tau_{{#1}}}
\newcommand{\me}{\mathscr{M}}
\newcommand{\mep}{\mathscr{M}^+}

\newcommand{\s}{\Sigma}
\renewcommand{\-}{\setminus}

\newcommand{\om}{\Omega}
\newcommand{\e}{\varepsilon}
\newcommand{\nv}{\norm{V}}

\newcommand{\nds}[1]{|\delta^s {#1}|}

\newcommand{\Xt}{\mathfrak{X}_t(\M)}

\newcommand{\Xc}{\mathfrak{X}_c(\M)}
\newcommand{\Xm}{\mathfrak{X}(\M)}
\newcommand{\X}{\mathfrak{X}}
\newcommand{\Xo}{\mathfrak{X}_0(\M)}
\newcommand{\Xp}{\mathfrak{X}_\perp(\M)}
\newcommand{\cc}{\subset \subset}
\newcommand{\ind}{\mathbf{1}}
\newcommand{\interior}[1]{%
  {\kern0pt#1}^{\mathrm{o}}%
}
\newcommand{\clos}[1]{\mkern 1mu\overline{\mkern-1mu#1\mkern-1mu}\mkern 1mu}

\newcommand{\pM}{\partial \M}

\newcommand{\T}{\Theta}
\newcommand{\Tn}{\Theta^k}
\newcommand{\tk}{\Theta^k}
\newcommand{\tkv}[1]{\Theta^k(\nv,{#1})}

\newcommand{\tuu}{\Theta^{* (k-1)}}
\newcommand{\tlu}{\Theta_*^{k-1}}
\newcommand{\tuk}{\Theta^{* k}}
\newcommand{\tlk}{\Theta_*^{k}}
\newcommand{\muv}{\sigma_V}
\newcommand{\mus}{\sigma_V^*}
\newcommand{\muc}{\sigma_C}
\newcommand{\muj}{\sigma_j}
\newcommand{\muvj}{\sigma_{V_j}}
\newcommand{\wto}{\stackrel{\ast}{\rightharpoonup}}

\newcommand*\dif{\mathop{}\!\mathrm{d}}
\DeclareMathOperator*{\wtoi}{\overset{\ast}{\rightharpoonup}}
\DeclareMathOperator*{\utoi}{\overset{unif}{\longrightarrow}}
\DeclareMathOperator*{\diam}{diam}

\DeclareMathOperator*{\hdim}{\mathcal{H}_{dim}}

\newcommand{\R}{\mathbb{R}}
\newcommand{\ru}{\mathbb{R}^{n}}
\newcommand{\N}{\mathbb{N}}

\newcommand{\M}{\mathcal{M}}

\newcommand{\W}{T_0^+ \M}

\newcommand{\V}{\mathcal{V}}

\renewcommand{\T}{\Theta}

\newcommand{\haus}[1]{\mathcal{H}^{#1}}

\numberwithin{equation}{section}
\newtheorem{theorem}{Theorem}[section]
\newtheorem*{theorem*}{Theorem}
\newtheorem{lemma}[theorem]{Lemma}
\newtheorem{corollary}[theorem]{Corollary}

\theoremstyle{definition}
\newtheorem{defi}{Definition}[section]
\newtheorem*{notation}{Notation}
\newtheorem{assumption}{Assumption}

\theoremstyle{remark}
\newtheorem{rem}{Remark}[section]
\newlist{steps}{enumerate}{1}
\setlist[steps, 1]{label = \textbf{Step \arabic*:}}
\newcounter{count}

\title{Rectifiability of the free boundary for varifolds}
\author{Luigi De Masi}
\begin{document}
\begin{abstract}
We establish a partial rectifiability result for the free boundary of a $k$-varifold $V$. Namely, we first refine a theorem of Gr\"uter and Jost by showing that the first variation of a general varifold with free boundary is a Radon measure.
Next we show that if the mean curvature $H$ of $V$ is in $L^p$ for some $p \in [1,k]$, then the set of points where the $k$-density of $V$ does not exist or is infinite has Hausdorff dimension at most $k-p$.
We use this result to prove, under suitable assumptions, that the part of the first variation of $V$ with positive and finite $(k-1)$-density is $(k-1)$-rectifiable.
\end{abstract}
\date{}
\maketitle

\section{Introduction}\label{sec:introduction}
\subsection{Motivations}
The main goal of this paper is to study the rectifiability of the free boundary for a $k$-varifold $V$ in a compact domain $\M \subset \R^n$ with smooth boundary $\pM$.

We say that $V$ has free boundary at $\pM$ if the following first variation formula holds for every vector field $X$ that is tangent to $\pM$ (see next section for more detailed definitions):
\begin{equation}\label{eq:fvf_free_boundary_intro}
\int_{G_k(\M)} \dive_S X(x) \dif V(x,S)
=
- \int_\M \scal{X}{H} \dif \nv,
\end{equation}
where $H \in L^1(\M,\nv)$.

If $V$ is the varifold induced by a smooth $k$-surface $\Sigma$ with smooth boundary $\partial \s$, $H$ is the mean curvature vector of $\s$ and \eqref{eq:fvf_free_boundary_intro} implies that
%$\Sigma$ has mean curvature $H$, that 
$\partial \Sigma \subset \pM$ and that $\Sigma$ meets $\pM$ orthogonally: that is the unit conormal $\eta$ to $\partial \Sigma$ coincides with the exterior unit normal vector $N$ to $\pM$. So, varifolds with free boundary generalize in a weak sense the idea of surfaces that meet $\pM$ orthogonally.

If $\s \subset \M$ is a smooth $k$-surface with smooth boundary $\partial \s$ and meets $\pM$ orthogonally, we can test the first variation formula
\begin{equation}\label{eq:fvf_smooth_intro}
\int_{\Sigma} \dive_{T_x \Sigma} X(x) \dif \haus{k}(x)
=
- \int_\Sigma \scal{X}{H} \dif \haus{k}
+ \int_{\partial \Sigma} \scal{X}{N} \dif \haus{k-1}
\end{equation}
(where $\haus{s}$ is the $s$-dimensional Hausdorff measure) with a smooth vector field $X$ such that $X(x)=N(x)$ on $\pM$, obtaining the estimate
\begin{equation}\label{eq:bound_smooth_boundary}
\haus{k-1}(\partial \s)
\leq
c \left( \frac{\haus{k}(\s)}{R(\M)} + \int_\s |H| \dif \haus{k} \right)
\end{equation}
where $c=c(k,\M)$ and $R(\M)$ is the minimum radius of curvature of $\pM$. This bound can be easily localized to any ball $B_r(x)$ where $x \in \pM$. (In particular, if $ \M = B_1$ is the unit ball with center 0 and $\s \subset B_1$ is a minimal $k$-surface that meets $\partial B_1$ orthogonally, choosing $X=x$ we obtain  the nice identity
$
\haus{k-1}(\partial \s)
=
k \haus{k}(\s).
$
)

The simple proofs of these a-priori bounds strongly rely on the fact that $\s$ and $\partial \s$ are assumed to be smooth (that is the first variation of $\s$ is assumed to be bounded) and on the assumption that the conormal $\eta$ of $\partial \s$ points outside $\M$. It is natural to ask if similar estimates hold also for a general varifold with free boundary $V$: that is if a varifold $V$ with free boundary at $\pM$ has bounded first variation and if its unit conormal on $\pM$ is orthogonal to $\pM$ and points outside $\M$.

We answer these questions refining a result stated by Gr\"uter and Jost in \cite{Gruter:allardtype} and by Edelen in \cite{Edelen:freeboundary}: we prove that if $V$ satisfies \eqref{eq:fvf_free_boundary_intro}, then it has bounded first variation: namely there exists a positive Radon measure $\muv$ on $\pM$  and a $\nv$-measurable vector field $\tilde{H}$ on $\pM$ such that, for every smooth vector field $X$ on $\M$ we have
\begin{equation}\label{eq:fvf_bounded_intro}
\int_{G_k(\M)} \dive_S X(x) \dif V(x,S)
=
- \int_\M \scal{X}{H + \tilde{H}} \dif \nv
+ \int_{\pM} \scal{X}{N} \dif \muv,
\end{equation}
where $N$ is the exterior unit normal vector to $\pM$ and $\tilde{H} \in L^\infty(\pM, \nv)$ is orthogonal to $\pM$. Moreover we prove bounds on $\muv$ similar to \eqref{eq:bound_smooth_boundary}. The measure $\tilde{H}\nv + N \muv$ is the orthogonal part of the first variation of $V$: $\tilde{H}\nv$ takes into account the absolutely continuous part (with respect to $\nv$), that is the areas where $V$ ``lean" on $\pM$ tangentially, while $\muv$ takes into account the boundary part of the first variation, that is where $V$ meets $\pM$ transversally.

Indeed, \eqref{eq:fvf_bounded_intro} is clearly analogous to \eqref{eq:fvf_smooth_intro}: by comparison we have that if $V$ is induced by a smooth surface $\s$ with smooth boundary $\partial \s$, then $\muv = \haus{k-1}\llcorner \partial \Sigma$.
It is then natural to ask also for a general $k$-varifold with free boundary $V$, if $\muv$ is singular with respect to $\nv$ or, more precisely, if $\muv$ is $(k-1)$-rectifiable. As far as we know, this question has not been investigated. Under suitable assumptions, we are able to show a rectifiability result for $\muv$.

To prove it, we analyze tangent cones to $V$ at points on $\pM$ to get informations about the tangent measures of $\muv$; tangent varifolds to $V$ exist if the upper $k$-density of $V$ is finite. When the mean curvature $H$ of $V$ is in $L^p(\M,\nv)$ for some $p>k$, it is well-known that the density of $V$ exists and is finite for every point; whereas if $p \leq k$, the finiteness of the upper density is guaranteed just $\nv$-a.e.\
% (and, by standard arguments, $\haus{k}$-a.e.)
by monotonicity formulae and differentiation theorems, which is not enough to prove any rectifiability result on $\muv$ (since $\pM$ may have $\nv$-measure $0$).

In order to deal with this case, we prove an estimate of the size of the set where the $k$-density of the varifold does not exists or is infinite, in terms of Hausdorff measures: if $H \in L^p(\M,\nv)$ for some $p \in [1,k]$, then this set has Hausdorff dimension at most $k-p$.

\subsection{Background and main results}
Allard studied the first variation of a varifold in the two seminal papers \cite{Allard1} and \cite{Allard2}. In the former he considered the interior case, while in the latter he studied the behavior of a $k$-varifold $V$ assuming it has, as a boundary, a smooth $(k-1)$-dimensional submanifold $\Gamma$, i.e. he assumes that $V$ has generalized mean curvature with respect to vector fields that vanish on $\Gamma$. An extension of the boundary result of Allard can be found in \cite{Bourni}.

An $\e$-regularity theorem similar to the ones by Allard is proved by Gr\"uter and Jost in \cite{Gruter:allardtype} for varifolds with free boundaries.
Moreover they prove in \cite[4.11(ii)]{Gruter:allardtype} that a varifold with free boundary with $\nv(\pM)=0$ has bounded first variation $\delta V$;
this is also proved by Edelen in \cite[Proposition 3.2]{Edelen:freeboundary} removing the hypothesis that $\nv(\pM)=0$, but assuming that $V$ is rectifiable.

We refine these boundedness results, extending them to general varifolds and removing the assumption \mbox{$\nv(\pM)=0$}.
We state it in a slightly more general setting: if $V$ has generalized mean curvature with respect to vector fields that vanish on $\pM$,
%(see Definition \ref{def:varifold_generalized_curvature}),
then it has bounded first variation with respect to vector fields that are orthogonal to $\pM$ (see next section for precise definitions):
\begin{theorem}\label{thm:positivity_singular_first_variation}
Let $V \in \V_k(\M)$ be a $k$-varifold with generalized mean curvature $H$ with respect to $\Xo$, with $H \in L^1(\M,\nv)$. Then there exists a positive Radon measure $\muv$ on $\pM$ and a $\nv$-measurable vector field $\tilde{H}$ on $\pM$ such that, for any $X \in \Xp$, it holds
\begin{equation}\label{eq:normal_first_variation}
\begin{split}
%\int_{G_k(\M)} \dive_S X(x) \dif V(x,S)
%=
%- \int_\M \scal{X}{\tilde{H}} \dif \nv
%+ \int_{\pM} \scal{X}{N} \dif \muv
%\qquad
%\forall X \in \Xp
\int_{G_k(\M)} \dive_S X(x) \dif V(x,S)
=
- \int_\M \scal{X}{H + \tilde{H}} \dif \nv
%- \int_\M \scal{X}{H} \dif \nv
%- \int_{G_k(\pM)} X(x) \dive_S N(x) \dif V(x,S)
+ \int_{\pM} \scal{X}{N} \dif \muv
\end{split}
\end{equation}
where
%$\tilde{H}\equiv H$ in $\interior{M}$,
$\tilde{H}$ is orthogonal to $\pM$ for $\nv$-a.e.\ $x \in\pM$, $\tilde{H} \in L^\infty(\pM, \nv)$ and $\normi{\tilde{H}}$ depends only on the second fundamental form of $\pM$.
In particular, $V$ has bounded first variation with respect to $\Xp$. Moreover, the following global and local estimates hold:
\begin{equation}\label{eq:estimate_muv_global}
\muv (\pM)
\leq
c \nv(\M) + \int_\M |H| \dif \nv;
\end{equation}
\begin{equation}\label{eq:estimate_measure_boundary}
\muv \big(B_{r/2}(x_0)\big)
\leq
\frac{c}{r} \nv \big( B_r(x_0) \big)
+
\int_{B_r(x_0)} |H| \dif \nv
\qquad
\forall x_0 \in \pM,
\,
\forall r \leq R(\M),
\end{equation}
where the constant $c=c(\M)$ depends only on the second fundamental form of $\pM$ and $R(\M)$ is such that the signed distance function from $\pM$ is of class $C^2$ in $U_R(\pM)$.
\end{theorem}
In other words, this theorem states that the component of the first variation of $V$ on $\pM$ that is orthogonal to $\pM$ is the sum of two terms:
\begin{itemize}
\item An absolutely continuous part with respect to $\nv$ given by $\tilde{H} \nv$, which takes into account the fact that $V$ can ``lean" on $\pM$: indeed,
if $V$ is induced by a smooth surface $\s$ with constant multiplicity, then $\tilde{H}$ is the mean curvature of $\s$ where $\s$ lean on $\pM$. This is orthogonal to $\pM$ and depends on $T_x \s$ and of the second fundamental form of $\pM$, see \eqref{eq:def_Htilde}.
\item A part given by $N \muv$, which ``points outward $\M$" and is bounded.
Roughly speaking, we expect that this is the ``transversal boundary" of $V$ at $\pM$. On the other hand, $V$ can have unbounded first variation on $\pM$ only where ``$V$ meets $\pM$ tangentially".
\end{itemize}
Looking at the case of varifolds with free boundary, since the tangent part to $\pM$ of the first variation of such a varifold is controlled by definition, Theorem \ref{thm:positivity_singular_first_variation} easily implies that varifolds with free boundary have bounded first variation (Corollary \ref{cor:bv_free_boundary_varifolds}).

As we said before, Corollary \ref{cor:bv_free_boundary_varifolds} was already proved by Gr\"uter and Jost when $\nv(\pM)=0$, and Edelen extended the result to $\nv(\pM)>0$ but assuming that $V$ is rectifiable. In Edelen's proof, the rectifiability is used to show that for $\nv$-a.e.\ $x \in \pM$ the only planes charged by $V$ are those included in $T_x \pM$.
We are able to remove the rectifiability assumption by the use of Lemma \ref{lmm:constancy}, which is a form of the Constancy Theorem (see \cite[Theorem 41.1]{simon:lectures}) and asserts that on $\pM$, even in the non-rectifiable case, $V$ charges only planes that are included in $T_x \pM$.

We have already stated above that, if the varifold is induced by a smooth surface with free boundary at $\pM$, then $\muv = \haus{k-1} \llcorner \partial \Sigma$.
It is then natural to ask if $\muv$ is $(k-1)$-rectifiable in a more general case as well.

In order to state precisely our main result concerning $\muv$, we define the \emph{$(k-1)$-dimensional part} of $\muv$, written $\mus$, as the restriction of $\muv$ to those points with strictly positive lower $(k-1)$-density and finite upper $(k-1)$-density :
\begin{equation}\label{eq:definition_mus}
\mus
=
{\muv}
\llcorner
E,
\qquad
E=
\{
x \mid
0< \tlu(\muv,x) \leq \tuu(\muv,x) < +\infty
\}.
\end{equation}
Our main result on $\mus$ is the following:
%(see next section for detailed definitions)
\begin{theorem}\label{thm:rectifiability_free_boundary}
Let $V \in \V_k(\M)$ be a rectifiable $k$-varifold with free boundary at $\pM$ such that $H \in L^p(\M,\nv)$ for some $p>1$ and $\tkv{x} \geq 1$ for $\nv$-a.e.\ $x \in \M$. Then $\mus$ is $(k-1)$-rectifiable.
\end{theorem}

To prove Theorem \ref{thm:rectifiability_free_boundary}, we make an analysis of the blow-ups of $V$ on $\pM$ similar to the one performed in \cite{DePDeRGhi16}. The study of blow-ups of $V$ allows us to deduce, for $\mus$-a.e.\ $x \in \pM$, that every $(k-1)$-blow-up of $\mus$ at $x$ is of the form $\beta \haus{k-1} \llcorner S$ for some $(k-1)$-dimensional plane $S$ and $\beta>0$. The Marstrand-Mattila Rectifiability Criterion (Theorem \ref{thm:M-M_criterion}), then implies that $\mus$ is $(k-1)$-rectifiable.

As we show at the beginning of section \ref{sec:proof_rectifiability}, if $H \in L^p(\M, \nv)$ for some $p>k$, then the condition $\tuu(\muv,x) < +\infty$ is not restrictive, since it holds for every point $x$, basically by \eqref{eq:estimate_measure_boundary} and by the monotonicity formula for $\nv$ that we prove in Corollary \ref{cor:monotonicity_formula}.
% and by an estimate on the mean curvature similar to \eqref{eq:estimate_curvature_blow-up}.
 Thus in this case, Theorem \ref{thm:rectifiability_free_boundary} reads as follows.
\begin{theorem}\label{thm:rectif_fb_p>k}
Let $V \in \V_k(\M)$ be a rectifiable $k$-varifold with free boundary at $\pM$ such that $H \in L^p(\M,\nv)$ for some $p>k$ and $\tkv{x} \geq 1$ for $\nv$-a.e.\ $x \in \M$. Then the restriction $\muv \llcorner \{x \mid \tlu(\muv,x)>0 \}$ is $(k-1)$-rectifiable.
\end{theorem}

To perform the analysis of the blow-ups of $V$ in the proof of Theorem \ref{thm:rectifiability_free_boundary}, we have to distinguish two cases:
\begin{itemize}
\item 
If $H \in L^p(\M,\nv)$ for some $p>k$, then the monotonicity formula (Corollary \ref{cor:monotonicity_formula}) assures the existence of the $k$-density $\tkv{\cdot} < +\infty$ and of tangent cones to $V$ at \emph{every} point in $\pM$.

\item The situation is more delicate when $p \leq k$, because the Lebesgue-Besicovitch differentation Theorem applied to the classical monotonicity identity for varifolds with bounded first variation guarantees (see e.g. \cite[Lemma 40.5]{simon:lectures}) the existence and finiteness of the $k$-density and of tangent cones just $\nv$-a.e.. Since $\nv(\pM)=0$ may hold true, it was in principle possible that $\tkv{x} = +\infty$ for every point in $\pM \cap \supp \nv$, and this would stop our analysis.
\end{itemize}
To overcome this difficulty, in subsection \ref{sec:proof_regular_set} we study more carefully
the set of points where the $k$-density of $V$ exists and is finite also when $p \leq k$.

The behavior of the density for points in $\interior{\M}$ was studied by Menne \cite[2.9-2.11]{isoperimetricmenne2009}: in that paper the author was interested in a lower bound for the lower density of the varifold with respect to Hausdorff measures; in particular he shows that, if the density has a lower bound $\nv$-a.e. and $H \in L^p(\M,\nv)$, then at $\haus{k-p}$-a.e. point, either the lower density still satisfies the lower bound, or it is equal to $0$.

Although the existence and finiteness of the density with respect to lower dimensional Hausdorff measures seems to be well-known at least for points in $\interior{\M}$, it does not appear in the literature; thus we report the result with its proof, both for points in $\interior{\M}$ and on $\pM$, since it can be useful for a future reference too.

More precisely, we define the \emph{density set} of $V$ (Definition \ref{def:regular_points_varifolds}), denoted by $\reg(V)$,
and we prove the following result.
\begin{theorem}\label{thm:singular_set_V_lower_dimensional}
Let $V \in \V_k(\M)$ be a varifold with free boundary at $\pM$ such that $H \in L^p(\M,\nv)$ for some $1\leq p \leq \infty$. Then
\begin{equation}\label{eq:dimension_singular_points}
\haus{s}\big(\M \- \reg(V)\big)=0
\qquad
\forall s > k-p
\end{equation}
and, for every $x_0 \in \reg(V)$ there exists an increasing function $\varphi_{x_0} \colon \R^+ \to \R^+$ 
such that
\begin{equation}\label{eq:monotonicity_phi}
\frac{\nv \big( B_r(x_0) \big)}{r^k}
\leq
\frac{\nv \big( B_t(x_0) \big)}{t^k}
+ \varphi_{x_0}(t)
\quad
\forall \, 0 < r < t
%< \zeta(x_0)
,
\qquad \quad
\lim_{t \to 0} \varphi_{x_0}(t)=0.
\end{equation}
In particular,
the $k$-density $\T^k(\nv,x_0)$ exists and is finite for every $x_0 \in \reg(V)$; moreover, the restrictions of $\tkv{\cdot}$ to $\reg(V) \cap \interior{\M}$ and to $\reg(V) \cap \pM$ are upper semi-continuous.
\end{theorem}
Thus the set of points where $\tkv{\cdot}$ is infinte or does not exist has Hausdorff dimension at most $k-p$.
$\reg(V)$ is the ``good set" where we are able to study the blow-ups of $\nv$ and $\mus$, to conclude the proof of Theorem \ref{thm:rectifiability_free_boundary}. We remark that, if $p >k$, then the theorem is the well-known existence of the density at \emph{every} point and is a straightforward consequence of Corollary \ref{cor:monotonicity_formula}.

Besides its application in the proof of Theorem \ref{thm:rectifiability_free_boundary}, Theorem \ref{thm:singular_set_V_lower_dimensional} is also interesting in itself, since it is a natural counterpart of the similar result for the set of Lebesgue points of a Sobolev function proved by Federer and Ziemer in \cite{federer1972lebesgue}: if $f \in W^{1,p}(\M)$ for $p \in [1,n]$ and if $\leb(f)$ is the set of Lebesgue points of $f$, then $\M \- \leb(f)$ has Hausdorff dimension at most $n-p$.

\subsection{Outline of the paper}

In section \ref{sec:notations} we recall the notations and the definitions used throughout the paper.

In section \ref{sec:proof_positive_normal_variation} we first prove Lemma \ref{lmm:constancy}, which concerns the behavior of $V$ on $\pM$; next we move on to the proof of Theorem \ref{thm:positivity_singular_first_variation}.

In section \ref{sec:consequences_bv} we describe some consequences of Theorem \ref{thm:positivity_singular_first_variation}: we notice that, if $k=n-1$, $\tilde{H}$ coincides with the mean curvature of $\pM$; next we adapt Theorem \ref{thm:positivity_singular_first_variation} to an other class of varifolds; we move on to varifolds with free boundary, proving that they have bounded first variation, establishing some monotonicity formulae and refining Lemma \ref{lmm:constancy} in this framework. Lastly, in subsection \ref{sec:proof_regular_set}, we prove Theorem \ref{thm:singular_set_V_lower_dimensional}.

In section \ref{sec:proof_rectifiability} we prove Theorem \ref{thm:rectifiability_free_boundary}; this is obtained by the study of tangent varifolds to $V$ at points in $\reg(V)$, to prove that $\mus$ satisfies the hypotheses of the Marstrand-Mattila Rectifiability Criterion.

\subsection{Acknowledgments} I am grateful to Professor Guido De Philippis for his invaluable help, and to Professor Ulrich Menne and Carlo Gasparetto for the useful discussions and comments about these topics.

\section{Notations}\label{sec:notations}
\subsection{Basic notations}
For a fixed orthonormal system of coordinates, we denote by $e_i$ the $i$-th coordinate unit vector. If $x \in \ru$, we denote by $B_r(x)$ the closed ball with center $x$ and radius $r$ and by $\omega_n$ the Lebesgue measure of the unit ball in $\R^n$. Moreover we set $B_r := B_r(0)$.
$c$ and $c'$ are generic positive constants, unless otherwise specified.
If $v \in \R^{n}$ we call $\trasl_v$ the translation
\begin{equation*}
\trasl_v: x \mapsto x + v.
\end{equation*}
For each $A \subset \ru$, we denote by $\ind_A$ the indicator function of $A$, by $\clos{A}$ and $\interior{A}$ respectively the closure and the interior of $A$ in the euclidean topology.
If $A \subset \ru$,
%we call $d_A(\cdot) = \dist(\cdot, A)$ the distance function from $A$. If
and $r>0$ we write $U_r(A)$ for the tubular neighborhood of $A$, i.e.
\begin{equation*}
U_r(A)
=
\bigcup_{x \in A} B_r(x).
%\{
%x \in \ru \mid \dist(x, A) < r
%\}.
\end{equation*}
We denote by $S$ a generic $k$-dimensional linear subspace (or $k$-plane) of $\ru$ and we write $S^\perp$ for the orthogonal complement of $S$ in $\R^n$. We denote by $P_S$ the orthogonal projection on $S$. If $X$ is a $C^1$ vector field, we call $\dive_S X$ the scalar product $P_S \cdot DX$. If $\tau_1, \dots, \tau_k$ is an orthonormal basis of $S$, by simple computations one has
\begin{equation}\label{eq:tangential_divergence_defi}
\dive_S X(x)=
\sum_{i=1}^k D_{\tau_i} \scal{X(x)}{\tau_i}.
\end{equation}
Throughout the paper, $\gamma \in C^{\infty}\big([0,\infty)\big)$ denotes a cut-off function such that
\begin{itemize}
\item $\gamma(t) = 1$ for each $t \in [0, \frac{1}{2}]$;
\item $\gamma(t)=0$ for each $t \geq 1$;
\item $\gamma'(t) \leq 0$ and $|\gamma'(t)| \leq 3$ for every $t \in \R$.
\end{itemize}
For each $r>0$ and $x \in \ru$ we consider the dilation map
\begin{equation}\label{eq:dilation_map}
\dil{x,r}(y)= \frac{1}{r}(y-x).
\end{equation}
If $\Gamma$ is a $C^1$ $k$-dimensional sub-manifold in $\R^n$ and $x \in \Gamma$, we write $T_x \Gamma$ for the tangent space to $\Gamma$ at $x$. We see $T_x \Gamma$ as an immersed $k$-plane in $\R^n$; more precisely, we see $T_x \Gamma$ as the blow-up of $\Gamma$ at the point $x$. If $\Gamma$ has non-empty boundary $\partial \Gamma$ of class $C^1$ and if $x \in \partial \Gamma$, we see $T_x \Gamma$ as containing $T_x \partial \Gamma$, which divide $T_x \Gamma$ into two half-spaces. We call these two parts $T_x^+ \Gamma$ and $T_x^- \Gamma$ and we set $T_x^+ \Gamma$ to be the blow-up at $x$ of the interior part of $\Gamma$.

We work on a compact domain $\M \subset \R^{n}$ with $C^2$ boundary $\pM$. We write $N(x)$ for the exterior unit normal vector to $\pM$ at $x$. In the following, $d$ denotes the signed distance function from $\pM$ such that $d>0$ in $\interior{\M}$, that is
\begin{equation}\label{eq:distance_function_pM}
d(x)
=
\begin{cases}
\inf \{ |x-y| \mid y \in \pM \}
&
\mbox{if } x \in \M
\\
- \inf \{ |x-y| \mid y \in \pM \}
&
\mbox{if } x \in \R^n \setminus \M.
\end{cases}
\end{equation}
Unless otherwise specified, we denote by $R=R(\M)>0$ a number such that $d$ is $C^2$ in $\clos{U_R(\pM)}$. Thus $\nabla d$ exists in $\clos{U_R(\pM)}$ and points inside $\M$.

We work with several classes of vector fields on $\M$, which we denote with the letter $\X$ with subscripts based on their behavior on $\pM$:
\begin{gather}\label{eq:vector_fields}
\Xm = C^1(\clos{\M},\R^{n+1}),
\qquad
\Xt = \{X \in \Xm \mid X(x) \in T_x \pM, \,\, \forall x \in \pM\},
\notag
\\
\Xp = \{X \in \Xm \mid X(x) \in (T_x \pM)^\perp ,\,\, \forall x \in \pM\}
\\
\Xo = \{X \in \Xm \mid X(x)=0 , \,\,\forall x \in \pM\}
\quad
\Xc = \{X \in \Xm \mid \supp X \cc \interior{\M}\}. \notag
\end{gather}
If $\Gamma$ is a $C^2$ submanifold of $\R^n$, by slight abuse of notation we write $\X_t(\Gamma)$ (respectively $\X_0(\Gamma)$) for the set of compactly supported $C^1$ vector fields on $\R^n$ that are tangent to $\Gamma$ (respectively that vanish on $\Gamma$).

\subsection{Measures, rectifiable sets}
If $A \subset \R^{n}$, $\me(A,\R^{m})$ is the space of $\R^m$-valued Radon measures on $A$ and $\mep(A)$ is the space of positive Radon measures on $A$. If $\mu \in \me(A,\R^m)$ we denote by $|\mu|$ the total variation measure of $\mu$. If $B \subset \ru$ is Borel, we write $\mu \llcorner B$ for the restriction of the measure $\mu$ to $B$.
If $A \subset \ru$ has non-empty interior, we endow $\me(A,\R^m)$ with the weak*-topology: i.e. we say that a sequence of Radon measures $\{\muj\}_j$ converges to $\mu$ ($\muj \wto \mu$) if
\begin{equation*}
\lim_{j \to \infty}
\int f \dif \muj
=
\int f \dif \mu
\qquad
\forall f \in C_c(A,\R^m).
\end{equation*}
If $\mu \in \mep(\ru)$, $x \in \ru$, $k \in \N$, we define the upper and the lower $k$-densities of $\mu$ at $x$:
\begin{equation*}
\Theta^{*k}(\mu,x)
=
\limsup_{r \to 0} \frac{\mu\big(B_r(x)\big)}{r^k}
\qquad
\Theta_*^{k}(\mu,x)
=
\liminf_{r \to 0} \frac{\mu\big(B_r(x)\big)}{r^k}.
\end{equation*}
If the above limits coincide, then we define the $k$-density of $\mu$ at $x$ as their common value, which we denote by $\T^k(\mu,x)$.
The $k$-singular set $\sing^k(\mu)$ is defined as
\begin{equation*}
\sing^k(\mu)=
\{x \in \M \mid \T^{* k}(\mu,x) = +\infty\}.
\end{equation*}
If $\mu \in \me(A,\R^m)$ and $f: A \to \R^N$ is proper, we define the push-forward $f_\# \mu$ of $\mu$ through $f$ as the Radon measure in $\me(\R^N,\R^m)$ defined by
\begin{equation*}
f_\# \mu (B)
=
\mu\big(f^{-1}(B) \big)
\qquad
\forall B \subset \R^N
\mbox{ Borel.}
\end{equation*}
If $\mu \in \mep(\R^n)$, we say that $\nu$ is a $k$-blow-up of $\mu$ at $x$ or a $k$-tangent measure to $\mu$ in $x$ if there exists a sequence $r_j \downarrow 0$ such that
\begin{equation}\label{eq:weak_convergence_tangent}
\mu_j
:=
\frac{1}{r_j^k} \big( \dil{x,r_j}\big)_\# \mu
\wto
\nu.
\end{equation}
We denote by $\tang^k(\mu,x)$ the (possibly empty) set of $k$-blow-ups of $\mu$ at the point $x$. If $\tuk(\mu,x)<\infty$ then, by Banach-Alaoglu Theorem, $\tang^k(\mu,x)$ is non-empty. If $\tlk(\mu,x) >0$, then every $k$-blow-up of $\mu$ at $x$ is non-trivial; indeed, if $\nu \in \tang^k(\mu,x)$ and $\mu_j \wto \nu$ as in \eqref{eq:weak_convergence_tangent}, then
\begin{equation}\label{eq:blow-up_non-trivial}
\nu(B_1)
\geq
\limsup_j \mu_j (B_1)
= \limsup_k \frac{\mu (B_{r_j}(x))}{r_j^k}
\geq
\tlk(\mu,x)
>0.
\end{equation}
For each $s>0$, we denote by $\haus{s}$ the $s$-dimensional Hausdorff measure and, if $A \subset \R^N$, $\hdim(A)$ denotes the Hausdorff dimension of $A$.
We say that a Borel set $M \subset \ru$ is $k$-rectifiable if there exist $M_0 \subset \ru$ with $\haus{k}(M_0)=0$ and a countable family of $C^1$ $k$-submanifolds $\{M_j\}_{j = 1}^\infty$ such that
\begin{equation*}
M
\subset
\bigcup_{i=0}^\infty M_j.
\end{equation*}
We say that a measure $\mu \in \mep(\ru)$ is $k$-rectifiable if there exist a $k$-rectifiable set $M$ and a positive function $\theta \in L_{\loc}^1(M, \haus{k})$ such that $\mu= \theta \haus{k}\llcorner M$. 
\subsection{Varifolds}
If $1 \leq k \leq n$ we call $G(k,n)$ the Grassmannian of the un-oriented $k$-dimensional linear subspaces (or $k$-planes) of $\R^{n}$. If $A\subset \R^{n}$ we denote by $G_k(A) := A \times G(k,n)$ the trivial Grassmannian bundle over $A$.

%If $S$ is a $k$-plane, we denote by $P_S$ the orthogonal projection matrix on $S$.

A \emph{$k$-varifold} on $A$ is a positive Radon measure on $G_k(A)$. We denote by $\V_k(A)$ the set of all $k$-varifolds on $A$ and we endow $\V_k(A)$ with the topology of the weak*-convergence of Radon measures, i.e. we say that $V_j \wto V$ if
\begin{equation*}
\lim_{j \to \infty} \int_{G_k(A)} \varphi(x,S) \dif V_j(x,S)
=
\int_{G_k(A)} \varphi(x,S) \dif V(x,S)
\qquad
\forall \varphi \in C_c(G_k(A)).
\end{equation*}
A $k$-rectifiable measure $\mu= \theta \haus{k}\llcorner M$ in $\R^m$ induces the $k$-varifold
\begin{equation*}
V = \theta \haus{k}\llcorner M \otimes \delta_{T_xM},
\end{equation*}
where $T_xM$ is the approximate tangent space of $M$ at $x$. A varifold that is induced by a rectifiable set is called a \emph{rectifiable varifold}. If the multiplicity function assumes only integer values, we say that the varifold is \emph{integer rectifiable}.
If $V$ is a $k$-varifold on $\om$, the \emph{mass} $\norm{V}$ (or \emph{total variation}) of $V$ is the positive Radon measure defined as
\begin{equation*}
\norm{V}(A) = V(G_k(A))
\qquad
\forall A \subset \om \mbox{ Borel}.
\end{equation*}
If $V$ is the $k$-varifold induced by the rectifiable measure $\mu= \theta \haus{k}\llcorner M$, then
\begin{equation*}
\norm{V}(B)
=
\int_B \theta(x) \dif \haus{k}(x).
\end{equation*}
By slight abuse of notation, we often denote $\supp \nv$ by $\supp V$.
If $\om \subset \ru$ is a domain, $V \in \V_k(\om)$ and if $\psi:\om \to \ru$ is a diffeomorphism, the \emph{push forward} $\psi_{\sharp}V$ of $V$ through $\psi$ is the varifold in $\V_k\big(\psi(\om)\big)$ such that, $\forall \varphi \in C_c\big(G_k(\psi(\om))\big)$,
\begin{equation}\label{eq:push-forward_varifolds}
\int_{G_k\left(\psi(\om)\right)} \varphi(y,T) \dif \psi_{\sharp}V(y,T)
=
\int_{G_k(\om)} J_S\psi(x) \varphi\big(\psi(x),\dif \psi_x(S)\big) \dif V(x,S),
\end{equation}
where $J_S \psi(x)$ is the Jacobian of $\psi$ relative to the $k$-plane $S$, i.e.
\begin{equation*}
J_S \psi(x) = \sqrt{\det\big( (\dif \psi_x)_{|S}^* \circ \big( (\dif \psi_x)_{|S} \big)}.
\end{equation*}
We notice that this \emph{is not} the push forward of measures previous defined (which is denoted by the different symbol $f_\# \mu$). In fact, the push forward of varifolds is defined in this way in order to ensure the validity of the area formula: indeed if $V$ is induced by a rectifiable set $M$, then $\psi_\sharp V$ is induced by $\psi(M)$.
If $V \in \V_k(\R^n)$, we say that $C \in \V_k(\R^n)$ is a \emph{blow-up} of $V$ at $x$ or a \emph{tangent varifold} to $V$ at $x$ if there exists a sequence of radii $r_j \downarrow 0$ such that
\begin{equation*}
(\dil{x,r_j})_\sharp V
\wto
C.
\end{equation*}
We write $\tang(V,x)$ for the set of tangent varifold to $V$ at the point $x$.

If $V \in \V_k(\M)$ and if $X \in \Xm$ the \emph{first variation} $\delta V(X)$ of $V$ with respect to $X$ is
\begin{equation*}
\delta V(X) = \left. \frac{\dif}{\dif t} \Big( \norm*{(\psi_t)_{\sharp} V}(\ru) \Big) \right|_{t=0}
\end{equation*}
where $\psi_t$ is the flow map of $X$ at the time $t$.
The following \emph{first variation formula} holds:
\begin{equation*}
\delta V(X) = \int_{G_k(\M)} \dive_S X(x) \, \dif V(x,S).
\end{equation*}
We now define the class of varifolds with \emph{bounded first variation}:
\begin{defi}
We say that a varifold $V$ has \emph{bounded first variation} in $\M$ if
\begin{equation}\label{eq:bounded_firsv_variation}
\sup \{|\delta V(X)| \mid X \in \Xm \,,
%\norm{X}_{L^\infty}\leq 1
\max|X| \leq 1
\} < + \infty.
\end{equation}
If \eqref{eq:bounded_firsv_variation} holds with a proper subset of $\Xm$(e.g. $\Xc$, $\Xo$...) in place of $\Xm$, we say that $V$ has bounded first variation with respect to this subset. 
\end{defi}
Therefore $V$ has bounded first variation if there exists $\delta V \in \me(\M,\ru)$ such that, for any $X \in \Xm$,
\begin{equation}
\delta V(X)
=
\int_\M \dive_S X(x) \dif V(x,S)
=
\int_\M \scal{X(x)}{\zeta(x)} \dif |\delta V|(x),
\end{equation}
where $\zeta$ is the polar vector of $\delta V$ with respect to $|\delta V|$.

If $V$ has bounded first variation, then by Lebesgue decomposition there exist $\nds{V} \in \mep(\M)$, a $\nds{V}$-measurable function $\eta: \M \to \R^n$ and a $\norm{V}$-measurable function $H: \M \to \R^n$ such that
\begin{equation*}
\delta V(X)= - \int_\M \scal{H}{X} \, \dif \norm{V} + 
\int_\M \scal{X}{\eta} \dif \nds{V}
\qquad
\forall X \in \Xm
\end{equation*}
where $\nds{V}$ is the singular part of $|\delta V|$ with respect to $\norm{V}$:
\begin{equation*}
\nds{V} = |\delta V|\llcorner Z
\qquad
Z=
\Big\{
x \in \M \mid \limsup_{r \to 0} \frac{|\delta V|(B_r(x))}{\norm{V}(B_r(x))} = +\infty
\Big\}.
\end{equation*}
Since the previous formula is similar to the corresponding one for smooth surfaces, we call $H$ the \emph{generalized mean curvature} of $V$, $\nds{V}$ the \emph{boundary measure} of $V$, the set $Z$ is the \emph{boundary} of $V$ and $\eta$ is the \emph{unit co-normal} of $V$.

We now define the classes of varifolds with \emph{generalized mean curvature}:
\begin{defi}\label{def:varifold_generalized_curvature}
We say that $V \in \V_k(\M)$ has \emph{generalized mean curvature} with respect to $\Xc$ (respectively $\Xo$) if there exists a $\nv$-measurable vector field $H \in L^1(\M,\nv)$ such that $H(x)=0$ for $\nv$-a.e.\ $x \in \pM$ and for any $X \in \Xc$ (respectively $\Xo$) the following formula holds:
\begin{equation}\label{eq:generalized_mean_curvature}
\int_{G_k(\M)} \dive_S X(x) \, \dif V(x,S)
=
- \int_\M \scal{H}{X} \, \dif \norm{V}.
%\qquad
%\forall X \in \Xc
%\mbox{ (respectively } \Xo)
\end{equation}
\end{defi}
\begin{rem}\label{rem:generalized_curvature_vanishes}
The assumption that $H(x)=0$ for $\nv$-a.e.\ $x \in \pM$ is important: without this hypothesis, the generalized mean curvature is not uniquely defined, that is any combination $H + K$ with $\supp K \subset \pM$ satisfies \eqref{eq:generalized_mean_curvature}. Without excluding this ambiguity, \eqref{eq:normal_first_variation} is no longer true because of the extra term $\int_{\pM} \scal{K}{X} \dif \nv$ on the right-hand side.
Similarly, if $V$ has bounded variation with respect to $\Xc$ or $\Xo$, if not otherwise specified we assume that the polar vector of $\delta V$ vanishes at $|\delta V|$-a.e. point on $\pM$.
\end{rem}

Thus $V$ has \emph{generalized mean curvature} with respect to $\Xc$ (respectively $\Xo$, $\Xt$) if it has bounded variation with respect to $\Xc$ (respectively $\Xo$, $\Xt$) and $\delta V$ has no singular part with respect to $\nv$ when we test with vector fields in $\Xc$ (respectively $\Xo$, $\Xt$).

\begin{defi}[Varifold with free boundary]
We say that $V \in \V_k(\M)$ has \emph{free boundary} at $\pM$ if there exists a $\nv$-measurable vector field $H \in L^1(\M,\nv)$ such that $H(x)$ is tangent to $\pM$ for $\nv$-a.e.\ $x \in \pM$ and such that \eqref{eq:generalized_mean_curvature} holds for every $X \in \Xt$.
\end{defi}

\begin{rem}\label{rem:fb_curvature_tangent}
As in Remark \ref{rem:generalized_curvature_vanishes}, the assumption that $H$ is tangent to $\pM$ $\nv$-a.e. is important because otherwise \eqref{eq:total_fvf_varifold_free_boundary} is no longer true: every sum $H + K$, with $K$ orthogonal to $\pM$, satisfies  \eqref{eq:generalized_mean_curvature} for any $X \in \Xt$, whereas when we test with non-tangent vector fields the presence of $K$ is relevant.
As above, when we say that $V$ has bounded variation with respect to $\Xt$, if not otherwise specified we assume that the polar vector of $\delta V$ is tangent to $\pM$ at $|\delta V|$-a.e. point on $\pM$.
%We see that the generalized mean curvature of a varifold $V$ is uniquely defined only at points where one can test with vector fields that point in all directions: for example, if $V$ has generalized mean curvature with respect to $\Xc$ or $\Xo$, then any combination $H + K$ with $\supp K \subset \pM$ satisfies the definition. Furthermore, if $V$ has free boundary at $\pM$, only the component of $H$ that is tangent to $\pM$ can be uniquely determined, since the test vector fields are tangent to $\pM$.
%
%To solve this ambiguity, we set that, if not otherwise specified, when we say that $V$ has generalized mean curvature with respect to $\Xc$ and $\Xo$, then we assume $H(x)= 0$ for $\nv$-a.e.\ $x \in \pM$; while if $V$ has free boundary at $\pM$, then we assume that $H(x)$ is tangent to $\pM$ for $\nv$-a.e.\ $x \in \pM$.
%
%If it is not otherwise specified, we make similar assumption if $V$ has bounded first variation $\delta V$ with respet to $\Xc, \Xo$ (the polar vector of $\delta V$ with respect to $|\delta V|$ is assumed to be $0$ at $|\delta V|$-a.e.\ point on $\pM$) and to $\Xt$ (the polar vector of $\delta V$ with respect to $|\delta V|$ is assumed to be tangent to $\pM$ at $|\delta V|$-a.e.\ point on $\pM$)
\end{rem}
\begin{rem}
As we mentioned in the introduction, a varifold with free bondary meets $\pM$ orthogonally in a weak sense: indeed, if $V$ is the varifold induced by a surface $\s$ with smooth boundary $\partial \s$, \eqref{eq:generalized_mean_curvature} implies that the conormal $\eta$ to $\partial \s$ is orthogonal to $\pM$.
\end{rem}

\section{Proof of Theorem \ref{thm:positivity_singular_first_variation}}\label{sec:proof_positive_normal_variation}
\subsection{Constancy Lemma}
The proof of the Theorem \ref{thm:positivity_singular_first_variation} is based on the following lemma, which is a form of the Constancy Theorem \cite[Theorem 41.1]{simon:lectures} with weaker hypotheses (and conclusion) and it is interesting in itself. The result is used in the proof of the Theorem \ref{thm:positivity_singular_first_variation} to deal with the case $\nv(\pM) >0$.
\begin{lemma}\label{lmm:constancy}
Let $\Gamma \subset \R^n$ be a $C^2$-hypersurface without boundary, let $V \in \V_k(\ru)$ have bounded first variation $\delta_0 V$ with respect to $\X_0(\Gamma)$ (that is vector fields that vanish on $\Gamma$).
Then
\begin{equation*}
V \big(
\{
(x,S) \in G_k(\R^n) \mid x \in \Gamma, S \not\subset T_x \Gamma
\}
\big)
=0.
\end{equation*}
\end{lemma}
\begin{proof}
By a simple covering argument it is enough to prove the result locally: that is that, for each $x_0 \in \Gamma$, there exists $r=r(x_0)>0$ and a ball $B_r(x_0)$ such that
\begin{equation*}
V \big(
\{
(x,S) \in G_k(\R^n) \mid x \in \Gamma \cap B_r(x_0), S \not\subset T_x \Gamma
\}
\big)
=0.
\end{equation*}
We fix $x_0 \in \Gamma$ and without loss of generality we can assume that $x_0=0$.

Since $\Gamma$ is locally-orientable, there exists $r'>0$ and a ball $B_{r'}$ such that $B_{r'} \setminus \Gamma$ is made of two connected components $D^+$ and $D^-$ separated by $\Gamma$.
Only in this proof, $d$ denotes a fixed one of the two signed distance function from $\Gamma$ in $B_\rho$, that is
\begin{equation*}
d(x)
=
\begin{cases}
\inf \{ |x-y| \mid y \in \Gamma \}
&
\mbox{if } x \in D^+
\\
- \inf \{ |x-y| \mid y \in \Gamma \}
&
\mbox{if } x \in D^-.
\end{cases}
\end{equation*}
Since $\Gamma$ is of class $C^2$, there exists $r \leq r'$ such that $d \in C^2\big(\clos{B_r}\big)$. We set $r(x_0)=r$. 

We recall that $\gamma$ is the cut-off function defined in section \ref{sec:notations}.
Since $\nabla d(x)$ is orthogonal to $\Gamma$ for every $x \in \Gamma$, we have $|P_S \nabla d(x)|^2 =0$ if and only if $S \subset T_x \Gamma$. Therefore, to get the conclusion, it is enough to prove that
\begin{equation}\label{eq:equivalent_conclusion_constancy}
\int_{G_k(\Gamma)} \gamma\Big( \frac{|x|}{r}\Big)|P_S \nabla d(x)|^2 \dif V(x,S) = 0.
\end{equation}
To do so, we test \eqref{eq:generalized_mean_curvature} with a suitable vector field $X \in \Xo$.
If $\rho< r$, we choose $X(x)=d(x) \gamma\Big( \frac{|x|}{r}\Big) \gamma\Big( \frac{d(x)}{\rho}\Big) \nabla d(x)$. We clearly have $X \in \Xo$ and
\begin{equation*}
\begin{split}
\dive_S X(x)
= &
\gamma\Big( \frac{|x|}{r}\Big) \gamma\Big( \frac{d(x)}{\rho}\Big)|P_S\nabla d(x)|^2
+
\frac{d(x)}{r} \gamma'\Big( \frac{|x|}{r}\Big)  \gamma\Big( \frac{d(x)}{\rho}\Big) \scal{\frac{x}{|x|}}{P_S\nabla d(x)}
\\
& +
\frac{d(x)}{\rho} \gamma\Big( \frac{|x|}{r}\Big) \gamma'\Big( \frac{d(x)}{\rho}\Big)
|P_S\nabla d(x)|^2
+
d(x) \gamma\Big( \frac{|x|}{r}\Big) \gamma\Big( \frac{d(x)}{\rho}\Big)
\dive_S \nabla d(x).
\end{split}
\end{equation*}
Since $V$ has bounded first variation $\delta_0 V = \zeta |\delta_0 V|$ with respect to $\X_0(\Gamma)$ (where $\zeta$ is the polar vector of $\delta_0 V$ with respect its total variation $|\delta_0 V|$), testing \eqref{eq:generalized_mean_curvature} with $X$ we obtain
\begin{equation}\label{eq:fvf_constancy_lemma}
\begin{split}
&
\left|\int_{G_k(\R^{n})} \gamma\Big( \frac{|x|}{r}\Big) \gamma\Big( \frac{d(x)}{\rho}\Big) |P_S\nabla d(x)|^2 \dif V(x,S)\right|
\\
& \qquad \quad \leq
\left|\int_{G_k(\R^{n})} \frac{d(x)}{r} \gamma'\Big( \frac{|x|}{r}\Big)  \gamma\Big( \frac{d(x)}{\rho}\Big) \scal{\frac{x}{|x|}}{P_S\nabla d(x)}
\dif V(x,S)\right|
\\
& \qquad \qquad +
\left|\int_{G_k(\R^{n})}\frac{d(x)}{\rho} \gamma\Big( \frac{|x|}{r}\Big) \gamma'\Big( \frac{d(x)}{\rho}\Big) |P_S\nabla d(x)|^2 \dif V(x,S)\right|
\\
& \qquad \qquad +
\left| \int_{G_k(\R^{n})} d(x) \gamma\Big( \frac{|x|}{r}\Big) \gamma\Big( \frac{d(x)}{\rho}\Big) \dive_S \nabla d(x) \dif V(x,S) \right|
\\
& \qquad \qquad +
\left| \int_{\R^{n}} \scal{X(x)}{\zeta(x)} \dif |\delta_0 V|(x) \right|
\end{split}
\end{equation}
For the left-hand side of the above inequality, by dominated convergence we have
\begin{equation*}
\lim_{\rho \to 0}
\int_{G_k(\R^{n})} \gamma\Big( \frac{|x|}{r}\Big) \gamma\Big( \frac{d(x)}{\rho}\Big) |P_S\nabla d(x)|^2 \dif V(x,S)
=
\int_{G_k(\Gamma)} \gamma\Big( \frac{|x|}{r}\Big) |P_S \nabla d(x)|^2 \dif V(x,S)
.
\end{equation*}
Therefore, to show \eqref{eq:equivalent_conclusion_constancy}, we have to prove that the terms on the right-hand side of \eqref{eq:fvf_constancy_lemma} go to $0$ as $\rho \to 0$:
\begin{enumerate}
\item
Since $\gamma'$ is bounded and $d(x) \leq \rho$ by cut-off, for the first term we have
\begin{equation*}
\begin{split}
\lim_{\rho \to 0}
\left|
\int_{G_k(\R^{n})}\frac{d(x)}{r} \gamma' \Big( \frac{|x|}{r}\Big) \gamma\Big( \frac{d(x)}{\rho}\Big) \scal{\frac{x}{|x|}}{P_S\nabla d(x)}\dif V(x,S)
\right|
\leq
c \lim_{\rho \to 0} \frac{\rho}{r} \nv \big( U_\rho(\Gamma) \cap B_r \big)
=
0.
\end{split}
\end{equation*}
\item Since $\bigl| \frac{d}{\rho} \bigr| \leq 1$ and $\gamma'(s) \neq 0$ only if $s \in (1/2,1)$, for the second term we have
\begin{equation*}
\begin{split}
&
\lim_{\rho \to 0}
\left|
\int_{G_k(\R^{n})}\frac{d(x)}{\rho} \gamma \Big( \frac{|x|}{r}\Big)\gamma'\Big( \frac{d(x)}{\rho}\Big) |P_S\nabla d(x)|^2 \dif V(x,S)
\right|
\\
& \qquad \qquad \leq
3 \lim_{\rho \to 0} \nv \Big( \big(U_\rho(\Gamma) \setminus U_{\rho/2}(\Gamma) \big) \cap B_r \Big)
=
0.
\end{split}
\end{equation*}
\item
%Since $\Gamma$ is $C^2$, it follows that $|\dive_S \nabla d(x)|$ is uniformly bounded by a constant $C$
By the choice of $r$ we have that $|\dive_S \nabla d(x)| \leq c$ in $B_r$; thus
\begin{equation*}
\begin{split}
\lim_{\rho \to 0}
\left|
\int_{G_k(\R^{n})}d(x) \gamma \Big( \frac{|x|}{r}\Big) \gamma\Big( \frac{d(x)}{\rho}\Big) \dive_S \nabla d(x) \dif V(x,S)
\right|
\leq
c \lim_{\rho \to 0} \rho \nv \big( U_\rho(\Gamma) \cap B_r \big)
=
0.
\end{split}
\end{equation*}
\item For the last term, since $\delta_0 V$ is a Radon measure, it holds
\begin{equation*}
\begin{split}
\lim_{\rho \to 0}
\left|
\int_{\R^{n}} \scal{X(x)}{\zeta(x)} \dif |\delta_0 V|(x)
\right|
\leq
\lim_{\rho \to 0}\rho |\delta_0 V|\big( U_\rho(\Gamma) \cap B_r \big)
=
0.
\end{split}
\end{equation*}
\end{enumerate}
This completes the proof.
\end{proof}

\subsection{Proof of Theorem \ref{thm:positivity_singular_first_variation}}
We can now prove Theorem \ref{thm:positivity_singular_first_variation}.

\begin{proof}[Proof of Theorem \ref{thm:positivity_singular_first_variation}]
Let us fix $R>0$ (as in section \ref{sec:notations}) so that the distance function $d$ from $\pM$ defined in \eqref{eq:distance_function_pM} is of class $C^2$ in $\clos{U_R(\pM)}$.

In what follows we are going to repeatedly use the decomposition of a vector field $X \in \Xm$ we now present;
within $U_R(\pM)$, we can decompose $X$ in its normal and tangent component: there exists a scalar function $\chi(x)$ such that $X= X^\perp + X^T$ with $X^\perp(x) = \chi(x) \nabla d(x)$ and $\scal{X^T(x)}{\nabla d(x)}=0$ for all $x \in U_R(\pM)$.

\begin{steps}[wide,%
%labelwidth=!,%
labelindent=0pt]
\item
We begin by cut-offing a vector field in its ``interior" and ``boundary" part.
For every $X \in \Xp$ and  $\rho < R$ one has
\begin{equation*}
X(x)
=
\gamma\Big( \frac{d(x)}{\rho}\Big) X(x)
+
\bigg( 1 - \gamma\Big( \frac{d(x)}{\rho}\Big) \bigg) X(x).
\end{equation*}
Therefore
\begin{equation}\label{eq:splitting_divergence}
\begin{split}
\int_{G_k(\M)} \dive_S X(x) \dif V(x,S)
= &
\int_{G_k(\M)} \dive_S \left[ \gamma\Big( \frac{d(x)}{\rho}\Big)X(x) \right] \dif V(x,S)
\\
&+
\int_{G_k(\M)} \dive_S \left[ \left( 1- \gamma\Big( \frac{d(x)}{\rho}\Big)\right) X(x) \right] \dif V(x,S).
\end{split}
\end{equation}
Thus we have splitted $X$ in the ``interior" and the ``boundary part" by cut-offing with $\gamma(d/\rho)$ and the idea is to send $\rho \to 0$.
\item\label{step:interior_dominated_convergence}
For the interior part,
since $\left( 1- \gamma\Big( \frac{d(x)}{\rho}\Big)\right) X(x) \in \Xc$ (i.e. it is compactly supported in the interior of $\M$), by \eqref{eq:generalized_mean_curvature} and dominated convergence we have
\begin{equation}\label{eq:compactly_supported_part}
\begin{split}
& \lim_{\rho \to 0}
\int_{G_k(\M)} \dive_S \left[ \left( 1- \gamma\Big( \frac{d(x)}{\rho}\Big)\right) X(x) \right]\dif V(x,S)
\\
& \qquad =
\lim_{\rho \to 0} \int_{\M} \left( 1- \gamma\Big( \frac{d(x)}{\rho}\Big)\right) \scal{X(x)}{H(x)} \dif \norm{V}(x)
\\
& \qquad =
- \int_{\interior{\M}} \scal{X(x)}{H(x)} \dif \norm{V}(x)
\\
& \qquad =
- \int_\M \scal{X(x)}{H(x)} \dif \norm{V}(x),
\end{split}
\end{equation}
where the last equality follows by the fact that $H$ is assumed to be equal to $0$ on $\pM$
%since it takes into account the first variation of $V$ with respect to vector fields that vanish on $\pM$
(see Remark \ref{rem:generalized_curvature_vanishes} for more details).
\item
Since the limit in \eqref{eq:compactly_supported_part} exists, also for the ``boundary part" of $X$ the limit
\begin{equation}\label{eq:limit_boundary_part_X}
\begin{split}
\lim_{\rho \to 0}
\int_{G_k(\M)}  \dive_S \left[ \gamma\Big( \frac{d(x)}{\rho}\Big)X(x) \right] \dif V(x,S)
\end{split}
\end{equation}
exists. We want now to compute \eqref{eq:limit_boundary_part_X} and to show the existence of $\tilde{H}$ and $\muv$ on $\pM$. We begin by writing
\begin{equation}\label{eq:divergence_at_boundary}
\begin{split}
\int \dive_S \left[ \gamma\Big( \frac{d(x)}{\rho}\Big)X(x)
\right] \dif V(x,S)
= &
\int
\gamma\Big( \frac{d(x)}{\rho}\Big) \dive_S X(x)
\dif V(x,S)
\\
& +
\int \frac{1}{\rho}\gamma'\Big( \frac{d(x)}{\rho}\Big)\scal{ P_S \nabla d(x)}{X(x)} \dif V(x,S) .
\end{split}
\end{equation}
At this point, we want to study separately the limit as $\rho \to 0$ of each term in the right-hand side of the above equation.

\begin{itemize}
\item
For the first term of the right-hand side in \eqref{eq:divergence_at_boundary},
we expect that, as $\rho \to 0$, it give a sort of mean curvature of $\pM$. This expectation is justified by the fact that, on $\pM$, $V$ charges only planes that are tangent to $\pM$ by Lemma \ref{lmm:constancy} and because $X$ is orthogonal to $\pM$ on $\pM$.
More precisely, we are going to prove that there exists a $\nv$-measurable vector field $\tilde{H}$ such that
\begin{equation} \label{eq:divergence_boundary_part1}
\begin{split}
\lim_{\rho \to 0}
\int_{G_k(\M)} \gamma\Big( \frac{d(x)}{\rho}\Big) \dive_S X(x) \dif V(x,S)
= 
- \int_{\pM} \scal{\tilde{H}(x)}{X(x)} \dif \nv(x),
\end{split}
\end{equation}
which is orthogonal to $\pM$ for $\nv$-a.e.\  $x \in \pM$.
To this aim, we first observe that by dominated convergence we have
\begin{equation}\label{eq:divergence_boundary_2}
\begin{split}
\lim_{\rho \to 0}
\int_{G_k(\M)} \gamma\Big( \frac{d(x)}{\rho}\Big) \dive_S X(x) \dif V(x,S)
=
\int_{G_k(\pM)} \dive_S X(x) \dif V(x,S)
%\\
%& \qquad =
%\int_{G_n(\pM)} \dive_S X^\perp (x) \dif V(x,S).
\end{split}
\end{equation}
We now have to compute the right-hand side of \eqref{eq:divergence_boundary_2}. To do this, we decompose $\dive_S X(x) = \dive_S X^\perp(x) + \dive_S X^T(x)$.
For the tangent part, we claim that
\begin{equation}\label{eq:dive_tangent_boundary}
\int_{G_k(\pM)} \dive_S X^T(x) \dif V(x,S)
=
0.
\end{equation}
In fact
$
\dive_S X^T(x) = 0
$
for any
$
x \in \pM,
$
and
$
\forall S \subset T_x \pM$. This is true because $X^T \equiv 0$ on $\pM$, therefore $D_\tau \scal{X(x)}{\tau}=0$ for any $\tau \in T_x \pM$. Thus \eqref{eq:dive_tangent_boundary} follows by definition of tangential divergence \eqref{eq:tangential_divergence_defi} and by Lemma \ref{lmm:constancy}.

For the orthogonal component we get
\begin{equation*}
\dive_S X^\perp(x)
=
\scal{P_S \nabla \chi}{\nabla d(x)}
+
\chi(x) \dive_S \nabla d(x).
\end{equation*}
Since $\nabla d(x)$ is orthogonal to $\pM$, by Lemma \ref{lmm:constancy} again we have $\scal{P_S\nabla \chi(x)}{ \nabla d(x)}=0$ for $V$-a.e.\ $(x,S) \in G_k(\pM)$. Hence
\begin{equation*}
\int_{G_k(\pM)} \scal{ P_S \nabla \chi(x)}{\nabla d(x)} \dif V(x,S)
=
0.
\end{equation*}
Since $\nabla d(x) = - N(x)$ for every $x \in \pM$, where $N(x)$ is the unit normal vector to $\pM$ at $x$, we obtain
\begin{equation}\label{eq:divergence_boundary_partial}
\int_{G_k(\pM)} \chi(x) \dive_S \nabla d(x) \dif V(x,S)
=
\int_{G_k(\pM)} \scal{X(x)}{N(x)} \dive_S  N(x) \dif V(x,S)
%- \int_{\pM} \scal{\tilde{H}(x)}{X(x)} \dif \nv(x)
\end{equation}
Thus, combining \eqref{eq:dive_tangent_boundary}-\eqref{eq:divergence_boundary_partial}, we obtain
\begin{equation}\label{eq:divergence_boundary_total}
\int_{G_k(\pM)} \dive_S X(x) \dif V(x,S)
=
%- \int_{\pM} \scal{\tilde{H}(x)}{X(x)} \dif \nv(x).
 \int_{G_k(\pM)} \scal{X(x)}{N(x)} \dive_S N(x) \dif V(x,S).
\end{equation}
We are now going to define $\tilde{H}$ and write the last integral in terms of it. To do so, by disintegration of $V$ we can write
\begin{equation*}
V = \nv \otimes \nu_x
\end{equation*}
where for $\nv$-a.e.\ $x$, $\nu_x$ is a probability measure on $G(k,n)$. Hence the right-hand side of \eqref{eq:divergence_boundary_total} can be written as
\begin{equation*}
\begin{split}
& \int_{G_k(\pM)} \scal{X(x)}{N(x)} \dive_S N(x) \dif V(x,S)
\\
& \qquad
=
 \int_{\pM} \scal{X(x)}{N(x)}
\left(
\int_{G(k,n)} \dive_S N(x) \dif \nu_x(S)
\right) \dif \nv(x).
\end{split}
\end{equation*}
If we define
\begin{equation}\label{eq:def_Htilde}
\tilde{H}(x)
:=
- N(x) \int_{G(k,n)} \dive_S N(x) \dif \nu_x(S)
\qquad
\mbox{ for }
\nv
\mbox{-a.e. }
x \in \pM,
%\tilde{H}(x)
%=
%\begin{cases}
%- N(x) \int_{G(k,n)} \dive_S N(x) \dif \sigma_x(S)
%&
%\mbox{ for }
%\nv
%\mbox{-a.e. }
%x \in \pM
%\\
%H(x)
%&
%\mbox{ for }
%\nv
%\mbox{-a.e. }
%x \in \interior{\M}.
%\end{cases}
\end{equation}
we can write \eqref{eq:divergence_boundary_total} as
\begin{equation}\label{eq:divergence_boundary_final_curvature}
\begin{split}
\int_{G_k(\pM)} \dive_S X(x) \dif V(x,S)
= 
- \int_{\pM} \scal{\tilde{H}(x)}{X(x)} \dif \nv(x).
\end{split}
\end{equation}
Loosely speaking, $\tilde{H}(x)$ can be interpreted as the ``mean curvature of $\pM$ weighted according to the planes charged by $V$ at $x$"; in fact, in co-dimension 1, $\tilde{H}$ turns out to be precisely the mean curvature of $\pM$ (see Corollary \ref{cor:positive_fv_codimension1}). By its definition, it is clear that $\tilde{H}$ is orthogonal to $\pM$ for $\nv$-a.e.\ $x \in \pM$, that $\tilde{H} \in L^\infty(\pM,\nv)$ and that $\norm{\tilde{H}}_{L^\infty(\pM,\nv)}$ depends only on the second fundamental form of $\pM$.
Gathering \eqref{eq:divergence_boundary_2} and \eqref{eq:divergence_boundary_final_curvature} we finally get \eqref{eq:divergence_boundary_part1}.

\item We now have to study the second term in the right-hand side of \eqref{eq:divergence_at_boundary}
Roughly speaking, we can see it as ``the mean orthogonal part to $\pM$ of $V$" on the tubular neighborhood $U_\rho(\pM)$. When $\rho \to 0$, we expect that this term takes into account the ``transversal boundary" of $V$ at $\pM$, that is the singular part of the first variation of $V$ on $\pM$.\\
More precisely, we are going to show the existence of a positive Radon measure $\muv$ (as expressed in the statement of the theorem), such that
\begin{equation}\label{eq:definition_mu_V}
\lim_{\rho \to 0}
\frac{1}{\rho}
\int \gamma'\Big( \frac{d(x)}{\rho}\Big)\scal{ P_S \nabla d(x)}{X(x)} \dif V(x,S)
=
\int_{\pM} \scal{X(x)}{N(x)} \dif \muv(x),
\end{equation}
where $N(x)$ is the exterior unit normal vector to $\pM$.
We first of all remark that the above limit exists by the existence of limits of the other two terms in \eqref{eq:divergence_at_boundary}.

To compute the limit in \eqref{eq:definition_mu_V}, we use again the decomposition $X = X^T + X^\perp = X^T + \chi \nabla d$.

As $\rho \to 0$, we expect that the contribution of $X^T$ is zero. Indeed, since $X \in \Xp$ and $X$ is of class $C^1$, there exists a constant $c>0$ such that $|X^T(x)| \leq c d(x)$. Therefore, since $\gamma'(s) \neq 0$ only for $s \in (1/2,1)$, we obtain
\begin{equation}\label{eq:limit_tangent_part_singular}
\begin{split}
\lim_{\rho \to 0} \,
&
\frac{1}{\rho}\int_{G_k(\M)} \left| \gamma'\Big( \frac{d(x)}{\rho}\Big)\scal{ P_S \nabla d(x)}{X^T(x)} \right| \dif V(x,S)
\\
& \qquad\qquad\qquad \leq
\lim_{\rho \to 0} 3 c \nv \big( U_\rho(\pM) \setminus U_{\rho/2}(\pM) \big)
=
0.
\end{split}
\end{equation}
This proves the existence of the limit for the orthogonal part
\begin{equation}\label{eq:existence_singular_variation}
\begin{split}
\lim_{\rho \to 0}
&
\frac{1}{\rho}\int_{G_k(\M)} \gamma'\Big( \frac{d(x)}{\rho}\Big) \scal{P_S \nabla d(x)}{X^\perp(x)} \dif V(x,S)
\\
& \qquad
=
\lim_{\rho \to 0}
\frac{1}{\rho}\int_{G_k(\M)} \chi(x)\gamma'\Big( \frac{d(x)}{\rho}\Big)|P_S \nabla d(x)|^2 \dif V(x,S).
\end{split}
\end{equation}
\eqref{eq:limit_tangent_part_singular} and \eqref{eq:existence_singular_variation} yield
\begin{equation}\label{eq:limit_definition_measure_boundary}
\begin{split}
& \lim_{\rho \to 0}
\int \frac{1}{\rho}\gamma'\Big( \frac{d(x)}{\rho}\Big)\scal{ P_S \nabla d(x)}{X(x)} \dif V(x,S)
\\
& \qquad
=
\lim_{\rho \to 0}
\frac{1}{\rho}\int_{G_k(\M)} \chi(x)\gamma'\Big( \frac{d(x)}{\rho}\Big)|P_S \nabla d(x)|^2 \dif V(x,S).
\end{split}
\end{equation}
We now observe that the map
\begin{equation*}
T \colon
X \in \Xp
\longmapsto
\lim_{\rho \to 0}
\frac{1}{\rho}\int_{G_k(\M)} \chi(x)\gamma'\Big( \frac{d(x)}{\rho}\Big)|P_S \nabla d(x)|^2 \dif V(x,S).
\end{equation*}
is a well-defined distribution and, by its definition, $\supp T \subseteq \pM$.
Again by definition, if $\chi(x) \leq 0$ for every $x \in \pM$ (i.e. if $X$ point outward $\pM$), then $T(X) \geq 0$. Therefore $T$ is a signed distribution and by Riesz Representation Theorem there exists a positive Radon measure $\muv$ such that
\begin{equation*}
\scal{T}{X} = \int_{\pM} \scal{X}{N} \dif \muv
\qquad
\forall X \in \Xp.
\end{equation*}
This completes the proof of \eqref{eq:definition_mu_V}.
\end{itemize}
\item
We now gather the previous computations to get \eqref{eq:normal_first_variation}.

By \eqref{eq:divergence_at_boundary}, \eqref{eq:divergence_boundary_part1} and \eqref{eq:definition_mu_V}, we can rewrite \eqref{eq:limit_boundary_part_X} as
\begin{equation}\label{eq:limit_boundary_part_final}
\begin{split}
\lim_{\rho \to 0}
\int_{G_k(\M)}  \dive_S \left[ \gamma\Big( \frac{d(x)}{\rho}\Big)X(x) \right] \dif V(x,S)
=
- \int_{\pM} \scal{\tilde{H}(x)}{X(x)} \dif \nv(x)
+
\int_{\pM} \scal{X}{N} \dif \muv.
%\lim_{\rho \to 0}
%\frac{1}{\rho}\int_{G_k(\M)} \chi(x)\gamma'\Big( \frac{d(x)}{\rho}\Big)|P_S \nabla d(x)|^2 \dif V(x,S).
\end{split}
\end{equation}
Going back to \eqref{eq:splitting_divergence}, by \eqref{eq:compactly_supported_part} and \eqref{eq:limit_boundary_part_final} we finally get
\begin{equation*}
\begin{split}
\int_{G_k(\M)} \dive_S X(x) \dif V(x,S)
=
- \int_\M \scal{X(x)}{H(x) + \tilde{H}(x)} \dif \nv(x)
+
\int_{\pM} \scal{X}{N} \dif \muv.
%\lim_{\rho \to 0}
%\frac{1}{\rho}\int_{G_k(\M)} \chi(x)\gamma'\Big( \frac{d(x)}{\rho}\Big)|P_S \nabla d(x)|^2 \dif V(x,S),
\end{split}
\end{equation*}
which completes the proof of \eqref{eq:normal_first_variation}.

\item
We are left with the proof of \eqref{eq:estimate_muv_global} and \eqref{eq:estimate_measure_boundary}. We begin with \eqref{eq:estimate_measure_boundary}. Let us fix $x_0 \in \pM$ and $r \leq R$.
Without loss of generality we can assume that $x_0=0$. To prove the estimate, we test \eqref{eq:normal_first_variation} with $X(x)=-\gamma\Big( \frac{|x|}{r}\Big) \nabla d(x)$ which clearly belongs to $\Xp$. We have
\begin{equation*}
\dive_S X(x)
=
- \gamma'\Big(\frac{|x|}{r}\Big) \frac{1}{r}
\scal{P_S \frac{x}{|x|}}{\nabla d(x)}
-
\gamma\Big( \frac{|x|}{r}\Big)
\dive_S \nabla d(x).
\end{equation*}
Therefore
\begin{equation}\label{eq:estimate_approximate_balls_b_measure}
\begin{split}
\muv \big( B_{r/2}(x) \big)
\leq &
\int_{\pM} \gamma\Big( \frac{|x|}{r}\Big) \dif \muv
\\
= &
\int_{\pM} \scal{X}{N} \dif \muv
\\
= &
\int_{G_k(\M)} \dive_S X(x) \dif V(x,S)
+
\int_\M \scal{X}{H + \tilde{H}} \dif \nv
\\
= &
\int_{G_k(\M)} \left[
- \gamma'\Big(\frac{|x|}{r}\Big) \frac{1}{r}
\scal{P_S \frac{x}{|x|}}{\nabla d(x)}
-
\gamma\Big( \frac{|x|}{r}\Big)
\dive_S \nabla d(x)
\right]
\dif V(x,S)
\\
& -
\int_\M \gamma\Big( \frac{|x|}{r}\Big) \scal{\nabla d(x)}{H(x) + \tilde{H}(x)} \dif \nv(x)
\end{split}
\end{equation}
We want to estimate the last member of the above inequality. To do so, we choose a constant $c=c(\pM,R)$ such that
\begin{equation*}
|\dive_S \nabla d(x)| \leq c
\qquad
\forall x \in U_R(\pM).
\end{equation*}
By \eqref{eq:def_Htilde}, the choice of $c$ provides also $|\tilde{H}(x)| \leq c$ for $\nv$-a.e.\ $x \in \pM$. Thus
\begin{equation*}
\begin{split}
-
\int_{G_k(\M)}
\gamma\Big( \frac{|x|}{r}\Big)
\dive_S \nabla d(x)
\dif V(x,S)
- &
\int_\M \gamma\Big( \frac{|x|}{r}\Big) \scal{\nabla d(x)}{H(x) + \tilde{H}(x)} \dif \nv(x)
\\
\leq &
\int_{B_r(x)} \big( c +|H(x)| \big) \dif \nv(x).
\end{split}
\end{equation*}
Substituting in \eqref{eq:estimate_approximate_balls_b_measure} and since $\gamma$ can be taken so that $\normi{\gamma'} \approx 2$, we get
\begin{equation*}
\begin{split}
\muv \big( B_{r/2}(x) \big)
\leq
\frac{c}{r} \nv \big( B_r(x) \big)
+
\int_{B_r(x)} |H(x)| \dif \nv(x).
\end{split}
\end{equation*}

The proof of \eqref{eq:estimate_muv_global} is similar to the previous one and is in fact easier: it is enough to take a vector field $X \in \Xm$ such that $X(x)= \gamma\Big( \frac{d(x)}{\rho}\Big) \nabla d(x)$ for some $\rho$ sufficiently small, so that $X=N$ on $\pM$ and use \eqref{eq:normal_first_variation}.
%Taking into account that $|\dive_S \nabla d| \leq c$, where $c$ depends on the second fundamental form of $\pM$, we obtain \eqref{eq:estimate_muv_global}.
\end{steps}
\end{proof}

\section{Consequences of Theorem \ref{thm:positivity_singular_first_variation}}
\label{sec:consequences_bv}
In this section we clarify some consequences of Theorem \ref{thm:positivity_singular_first_variation}.
\begin{itemize}
\item In subsection \ref{subsec:bv_vanishing} we extend Theorem \ref{thm:positivity_singular_first_variation} to the case of varifolds with bounded first variation with respect to $\Xo$;
\item
In subsection \ref{subsec:codimension1} we study the codimension $1$ case: we prove that if $k=n-1$, then the vector field $\tilde{H}$ given by Theorem \ref{thm:positivity_singular_first_variation} is the mean curvature vector of $\pM$ (Corollary \ref{cor:positive_fv_codimension1}); next we state a refined version of Lemma \ref{lmm:constancy} for varifolds with free boundary: if $k=n-1$, then the restriction of $V$ to $\pM$ is $(n-1)$-rectifiable (Corollary \ref{cor:constancy_free_boundary});
\item In subsection \ref{subsec:compact_supp}
we extend Theorem \ref{thm:positivity_singular_first_variation} to varifolds with generalized mean curvature with respect to $\Xc$, i.e. vector fields compactly supported in the interior of $\M$ (see \eqref{eq:vector_fields}), assuming that $\nv(\pM)=0$ (Corollary \ref{cor:positivity_compact}); As a consequence, such varifolds have generalized mean curvature with respect to the larger class of vector fields $\Xo$;
\item In subsection \ref{subsec:free_boundary} we show that varifolds with free boundary have bounded first variation (Corollary \ref{cor:bv_free_boundary_varifolds});
\item In subsection \ref{subsec:monotonicity_formulae} we prove a monotonicity inequality (Lemma \ref{lmm:monotonicity_inequality_s}) for points on $\pM$; we next use the inequality to obtain monotonicity formulae for points on $\pM$ (Corollaries \ref{cor:monotonicity_formula} and \ref{cor:monotonicity_formula_infty}) without the reflections used by Gr\"uter and Jost in \cite{Gruter:allardtype}.
\item In subsection \ref{sec:proof_regular_set} we use the monotonicity inequality at the boundary to prove Theorem \ref{thm:singular_set_V_lower_dimensional}.
\end{itemize}

\subsection{Varifolds with bounded variation with respect to $\Xo$}\label{subsec:bv_vanishing}
The fact that the first variation of $V$ with respect to $\Xo$ is absolutely continuous with respect to $\nv$ is not essential to prove Theorem \ref{thm:positivity_singular_first_variation}. In fact the following slight modification holds true.
\begin{theorem}\label{thm:bv_vanishing_bv}
Let $V \in \V_k(\M)$ be a $k$-varifold with bounded first variation $\delta_0 V$ with respect to $\Xo$. Then there exists a positive Radon measure $\muv$ on $\pM$ and a $\nv$-measurable vector field $\tilde{H}$ on $\pM$ such that, for any $X \in \Xp$, it holds
\begin{equation}\label{eq:normal_first_variation_bv}
\begin{split}
%\int_{G_k(\M)} \dive_S X(x) \dif V(x,S)
%=
%- \int_\M \scal{X}{\tilde{H}} \dif \nv
%+ \int_{\pM} \scal{X}{N} \dif \muv
%\qquad
%\forall X \in \Xp
\int_{G_k(\M)} \dive_S X(x) \dif V(x,S)
=
\int_\M \scal{X}{\zeta} \dif |\delta_0 V|
- \int_\M \scal{X}{\tilde{H}} \dif \nv
%- \int_\M \scal{X}{H} \dif \nv
%- \int_{G_k(\pM)} X(x) \dive_S N(x) \dif V(x,S)
+ \int_{\pM} \scal{X}{N} \dif \muv
\end{split}
\end{equation}
where
%$\tilde{H}\equiv H$ in $\interior{M}$,
$\tilde{H}$ is orthogonal to $\pM$ for $\nv$-a.e.\ $x \in\pM$, $\tilde{H} \in L^\infty(\pM, \nv)$ and $\normi{\tilde{H}}$ depends on the second fundamental form of $\pM$ and  $\zeta$ is the polar vector of $\delta_0 V$ with respect to $|\delta_0 V|$.
In particular, $V$ has bounded first variation with respect to $\Xp$. Moreover, the following estimates hold:
\begin{equation}\label{eq:estimate_muv_global_bv}
\muv(\pM)
\leq
c \nv(\M) + |\delta_0 V|(\M);
\end{equation}
\begin{equation}\label{eq:estimate_measure_boundary_bv}
\muv \big(B_{r/2}(x_0)\big)
\leq
\frac{c}{r} \nv \big( B_r(x_0) \big)
+
|\delta_0 V|\big(B_r(x_0) \big)
\qquad
\forall x_0 \in \pM,
\,
\forall r \leq R(\M)
\end{equation}
where $R(\M)$ is such that the distance function from $\pM$ is of class $C^2$ in $U_R(\pM)$ and the constant $c=c(\M)$ depends on the second fundamental form of $\pM$.
\end{theorem}
\begin{proof}
The proof follows the one of Theorem \ref{thm:positivity_singular_first_variation}. The only part of the proof of \eqref{eq:normal_first_variation} where we use the hypothesis $\delta_0 V \ll \nv$ is in \eqref{eq:compactly_supported_part}. If $V$ has bounded first variation $\delta_0 V$ with respect to $\Xo$, then one easily obtain
\begin{equation}\label{eq:compactly_supported_part_bv}
\begin{split}
& \lim_{\rho \to 0}
\int_{G_k(\M)} \dive_S \left[ \left( 1- \gamma\Big( \frac{d(x)}{\rho}\Big)\right) X(x) \right]\dif V(x,S)
\\
& \qquad =
\lim_{\rho \to 0} \int_{\M} \left( 1- \gamma\Big( \frac{d(x)}{\rho}\Big)\right) \scal{X(x)}{\zeta(x)} \dif |\delta_0 V|(x)
\\
& \qquad =
- \int_{\interior{\M}} \scal{X(x)}{\zeta(x)} \dif |\delta_0 V|(x)
\\
& \qquad =
- \int_\M \scal{X(x)}{\zeta(x)} \dif |\delta_0 V|(x)
\end{split}
\end{equation}
where $\zeta$ is the polar vector of $\delta_0 V$ with respect to $|\delta_0 V|$ and the last equality is due to the assumption $|\delta_0 V|(\pM)=0$
%since $\delta_0 V$ takes into account the first variation of $V$ with respect vector fields that vanish at $\pM$,
(see Remark \ref{rem:generalized_curvature_vanishes}).

The modifications to the proofs of \eqref{eq:estimate_muv_global} and \eqref{eq:estimate_measure_boundary} to obtain \eqref{eq:estimate_muv_global_bv} and \eqref{eq:estimate_measure_boundary_bv} are obvious.
\end{proof}

\subsection{The codimension 1 case: $k=n-1$}\label{subsec:codimension1}
If $k=n-1$, we can characterize $\tilde{H}$ in a simpler way: if $x \in \pM$, then $ \tilde{H}(x)$ is the mean curvature vector of $\pM$.
\begin{corollary}\label{cor:positive_fv_codimension1}
Let $V \in \V_{n-1}(\M)$ have generalized mean curvature $H$ with respect to $\Xo$ with $H \in L^1(\M,\nv)$.  Then there exists a positive Radon measure $\muv$ on $\pM$
%and a $\nv$-measurable function $\tilde{H}$
such that
\begin{equation*}
\begin{split}
\int_{G_{n-1}(\M)} \dive_S X(x) \dif V(x,S)
=
- \int_\M \scal{X}{H + \tilde{H}} \dif \nv
+ \int_{\pM} \scal{X}{N} \dif \muv
\qquad
\forall X \in \Xp
\end{split}
\end{equation*}
where $\tilde{H}$ is the mean curvature vector of $\pM$, that is $\tilde{H}(x)
:=
- N(x) \big(\dive_{T_x \pM} N(x) \big)$ for $x \in \pM$.
Moreover, \eqref{eq:estimate_measure_boundary} holds true.
\end{corollary}
\begin{proof}
If $S$ is an $(n-1)$-dimensional subspace of $\R^n$, then $S \subset T_x \pM$ if and only if $S=T_x \pM$. Therefore, if $k=n-1$ and $V \in \V_{n-1}(\M)$ with generalized mean curvature with respect to $\Xo$, Lemma \ref{lmm:constancy} yields
\begin{equation}\label{eq:constancy_codimension_1}
V \left(
\{
(x,S) \in G_{n-1}(\pM) \mid S \neq T_x \Gamma
\}
\right)
=0.
\end{equation}
Hence, for $X \in \Xp$,
\begin{equation*}
\int_{G_{n-1}(\pM)} X(x) \dive_S N(x) \dif V(x,S)
=
\int_{\pM} X(x) \dive_{T_x \pM} N(x) \dif \nv,
\end{equation*}
thus \eqref{eq:divergence_boundary_total} becomes
\begin{equation*}
\int_{G_{n-1}(\pM)} \dive_S X(x) \dif V(x,S)
=
\int_{\pM} \scal{N(x)}{X(x)}\dive_{T_x \pM} N(x) \dif \nv(x).
\end{equation*}
Thus, defining
\begin{equation*}
\tilde{H}(x)
:=
- N(x) \dive_{T_x \pM} N(x),
\end{equation*}
we get \eqref{eq:divergence_boundary_final_curvature}. 
\end{proof}

If $k = n-1$ and $V$ has bounded first variation with respect to $\Xt$ instead of $\Xo$, we can strengthen the conclusion of Lemma \ref{lmm:constancy}: $V \llcorner G_{n-1}(\pM)$ is $(n-1)$-rectifiable.
\begin{corollary}\label{cor:constancy_free_boundary}
Let $V \in \V_{n-1}(\M)$ with bounded first variation with respect to $\Xt$. Then $V \llcorner G_{n-1}(\pM)$ is an $(n-1)$-rectifiable varifold. More precisely, if $\varphi \in C_c\big( G_{n-1}(\M) \big)$, then
\begin{equation*}
\int_{G_{n-1}(\M)} \varphi(x,S) \dif V(x,S)
=
\int_{\pM} \varphi(x,T_x \pM) \theta(x) \dif \haus{n-1}(x)
+
\int_{G_{n-1}(\interior{\M})} \varphi(x,S) \dif V(x,S),
\end{equation*}
where $\theta(x)= (\omega_{n-1})^{-1}\T^{n-1}(\nv,x)$ for $\haus{n-1}$-a.e.\ $x \in \pM$.
\end{corollary}
\begin{proof}
By Lemma \ref{lmm:constancy} and in particular by \eqref{eq:constancy_codimension_1}, for any $\varphi \in C_c\big( G_{n-1}(\M) \big)$ we have
\begin{equation*}
\int_{G_{n-1}(\M)} \varphi(x,S) \dif V(x,S)
=
\int_{\pM} \varphi(x,T_x \pM) \dif \nv(x)
+
\int_{G_{n-1}(\interior{\M})} \varphi(x,S) \dif V(x,S).
\end{equation*}
By \cite[Lemma 40.5]{simon:lectures}), for $\nv$-a.e.\ $x \in \pM$ there exists the density $\T^{n-1}(\nv,x) < \infty$. Hence, since $\pM$ is of class $C^2$, the quantity
\begin{equation*}
\begin{split}
\theta(x)
:=
\lim_{\rho \to 0} \frac{\nv \big(B_\rho(x) \big)}{\haus{n-1}\llcorner \pM \big( B_\rho(x)\big)}
=
\frac{\T^{n-1}(\nv,x)}{\omega_{n-1}}
\end{split}
\end{equation*}
exists and is finite for $\nv$-a.e.\ $x \in \pM$.
By Radon-Nikodym Theorem \cite[Theorem 4.7]{simon:lectures}, since the singular set
$
\{
x \in \M \mid  \theta(x)= +\infty
\}
$ of $\nv \llcorner \pM$ is $\nv$-negligible, we have $\nv \llcorner \pM \ll \haus{n-1}\llcorner \pM$ and the conclusion follows.
\end{proof}

\subsection{Varifolds with mean curvature with respect to $\Xc$}\label{subsec:compact_supp}
The analogous of Theorem \ref{thm:positivity_singular_first_variation} holds, adding the extra hypothesis $\nv(\pM)=0$, even if $V$ has generalized mean curvature with respect to $\Xc$, i.e. the vector fields with compact support in the interior of $\M$. In fact, if we analyze the proof of Theorem \ref{thm:positivity_singular_first_variation}, we can see that the only point where we used the existence of generalized mean curvature with respect $\Xo$ is to obtain \eqref{eq:divergence_boundary_total}, that is to obtain the identity
\begin{equation*}
\int_{G_k(\pM)} \dive_S X(x) \dif V(x,S)
=
- \int_{G_k(\pM)} \scal{X(x)}{N(x)} \dive_S N(x) \dif V(x,S)
\end{equation*}
by the use of Lemma \ref{lmm:constancy}, whereas if $\nv(\pM)=0$ then obviously we have
\begin{equation*}
\int_{G_k(\pM)} \dive_S X(x) \dif V(x,S)
=
0.
\end{equation*}
Since the remaining arguments remain valid also if $V$ has mean curvature with respect to $\Xc$, we have proved the following corollary.
\begin{corollary}\label{cor:positivity_compact}
Let $V \in \V_k(\M)$ with generalized mean curvature $H$ with respect to $\Xc$ with $H \in L^1(\M,\nv)$ and $\nv(\pM)=0$. Then there exists a positive Radon measure $\muv$ on $\pM$ such that
\begin{equation}\label{eq:normal_compact}
\int_{G_k(\M)} \dive_S X(x) \dif V(x,S)
=
- \int_\M \scal{X}{H} \dif \nv
+ \int_{\pM} \scal{X}{N} \dif \muv
\qquad
\forall X \in \Xp
\end{equation}
In particular, $V$ has bounded first variation with respect to $\Xp$ and the estimates \eqref{eq:estimate_muv_global}, \eqref{eq:estimate_measure_boundary} on $\muv$ hold true.
\end{corollary}
\begin{rem}
If we remove the hypothesis $\nv(\pM)=0$ nothing can be said about the behavior of $V$ on $\pM$, because any vector field in $\Xc$ has first derivatives compactly supported in the interior of $\M$. So we have a lack of test vector fields to establish any property of $V$ on $\pM$: e.g. take a smooth surface in $\interior{\M}$ and add any varifold $W \in \V_k(\pM)$ with unbounded first variation with respect $\Xp$.
\end{rem}
Since $\Xo \subset \Xp$, the following result follows.
\begin{corollary}
Let $V \in \V_k(\M)$ with generalized mean curvature $H$ with respect to $\Xc$, with $H \in L^1(\M,\nv)$ and $\nv(\pM)=0$. Then $V$ has generalized mean curvature $H$ with respect to $\Xo$.
\end{corollary}

\subsection{Varifolds with free boundaries}\label{subsec:free_boundary}
As an immediate corollary of Theorem \ref{thm:positivity_singular_first_variation}, varifolds with free boundaries have bounded first variation.
\begin{corollary}\label{cor:bv_free_boundary_varifolds}
Let $V \in \V_k(\M)$ have free boundary at $\pM$ with $H \in L^1(\M,\nv)$. Then $V$ has bounded first variation. More precisely, there exists a positive Radon measure $\muv$ on $\pM$ and a $\nv$-measurable vector field $\tilde{H}$ such that
\begin{equation}\label{eq:total_fvf_varifold_free_boundary}
\begin{split}
\int_{G_k(\M)} \dive_S X(x) \dif V(x,S)
=
- \int_\M \scal{X}{H + \tilde{H}} \dif \nv
%- \int_{G_k(\pM)} X^\perp(x) \dive_S N(x) \dif V(x,S)
+ \int_{\pM} \scal{X}{N} \dif \muv
\qquad
\forall X \in \Xm,
\end{split}
\end{equation}
where $\tilde{H}$ is defined as in \eqref{eq:def_Htilde}. In particular $\tilde{H}$ is orthogonal to $\pM$, $\tilde{H} \in L^\infty(\pM, \nv)$ and $\normi{\tilde{H}}$ depends only on the second fundamental form of $\pM$.
Moreover, \eqref{eq:estimate_muv_global} \eqref{eq:estimate_measure_boundary} hold true.
\end{corollary}
\begin{proof}
If $X \in \Xm$, there exist $X^T, X^\perp$ such that $X=X^T+X^\perp$, with $X^T \in \Xt$ and $X^\perp \in \Xp$ (see the decomposition at the beginning of the proof of Theorem \ref{thm:positivity_singular_first_variation}). We have
\begin{equation*}
\dive_S X(x)
=
\dive_S X^T(x)
+
\dive_S X^\perp(x).
\end{equation*}
Since the mean curvature of $V$ with respect to $\Xo$ is given by $H \ind_{\interior{\M}}$, by Theorem \ref{thm:positivity_singular_first_variation} there exists a positive Radon measure $\muv$ on $\pM$ and a $\nv$-measurable vector field $\tilde{H}$ such that
\begin{equation*}
\begin{split}
\int_{G_k(\M)} \dive_S X^\perp(x) \dif V(x,S)
= &
- \int_\M \scal{X^\perp}{H \ind_{\interior{\M}} + \tilde{H}} \dif \nv
%- \int_{G_k(\pM)} X^\perp(x) \dive_S N(x) \dif V(x,S)
+ \int_{\pM} \scal{X}{N} \dif \muv
\\
=&
-
\int_\M \scal{X^\perp}{H + \tilde{H}} \dif \nv
+ \int_{\pM} \scal{X}{N} \dif \muv,
\end{split}
\end{equation*}
where the last equality follows because $H$ is assumed to be tangent to $\pM$ on $\pM$ (see Remark \ref{rem:fb_curvature_tangent}).
Moreover we have the same estimates
\eqref{eq:estimate_muv_global} and \eqref{eq:estimate_measure_boundary} on $\muv$. For what concerns the tangent part $X^T$, the definition of varifold with free boundary yields
\begin{equation*}
\begin{split}
\int_{G_k(\M)} \dive_S X^T(x) \dif V(x,S)
=
- \int_\M \scal{X^T}{H} \dif \nv.
\end{split}
\end{equation*}
This shows the conclusion.
\end{proof}

\subsection{Monotonicity formulae}\label{subsec:monotonicity_formulae}
Gr\"uter and Jost established in \cite{Gruter:allardtype} several properties of varifolds with free boundaries: monotonicity formulae for $\nv$ at the boundary \cite[Theorem 3.1]{Gruter:allardtype}, which imply the existence of $\T^k(\nv,x)$ for every point $x$ if the mean curvature $H \in L^p(\M,\nv)$ for some $p>k$.

The monotonicity results are obtained by reflecting the balls across $\pM$, i.e. they have monotonicity of the sum of the masses in the ball and in the reflected ball \cite[Theorem 3.1]{Gruter:allardtype}. Using Corollary \ref{cor:bv_free_boundary_varifolds} it is possible to obtain the monotonicity of the mass in $B_r(x)$, without reflecting the balls.

We begin with a monotonicity inequality which is used also in the proof of Theorem \ref{thm:singular_set_V_lower_dimensional}.
\begin{lemma}[Monotonicity inequality]\label{lmm:monotonicity_inequality_s}
Suppose $V \in \V_k(\M)$ has free boundary at $\pM$, with $H \in L^p(\M,\nv)$ for some $p \in [1,+\infty)$. Then there exists a constant $c>0$ that depends only on $n,k,p$ and on the second fundamental form of $\pM$ such that, for all $x_0 \in \pM$ and $s \in \R$ the following inequality holds:
\begin{equation}\label{eq:monotonicity_inequality_s}
\begin{split}
(1+c\rho)
\frac{\dif}{\dif \rho}
\left(
\frac{1}{\rho^k}
\int_{\M} \gamma\Big( \frac{|x-x_0|}{\rho}\Big)
\dif \nv
\right)^{\frac{1}{p}}
\geq &
- \rho^{-\frac{k-s}{p}}
\Big(
\frac{1}{p} + c\rho
\Big)
\left(
\frac{1}{\rho^s}
\int_{B_\rho} |H +\tilde{H}|^p \dif \nv
\right)^{\frac{1}{p}}
\\
& -
c(1+\rho)
\left( \frac{1}{\rho^k}
\int_{\M} \gamma\Big( \frac{|x-x_0|}{\rho}\Big) \dif \nv(x)
\right)^{\frac{1}{p}}
\end{split}
\end{equation}
\end{lemma}
\begin{proof}
Without loss of generality we can suppose $x_0 = 0$. Since for large $\rho$ the statement is obvious, we have to prove it only for $0<\rho< R(\M)$, where $R(\M)$ is defined in section \ref{sec:notations}. We want to bound from below the following derivative:
\begin{equation}\label{eq:monotonicity_derivative_initial}
\frac{\dif}{\dif \rho}
\left(
\frac{1}{\rho^k}
\int_{\M} \gamma\Big( \frac{|x|}{\rho}\Big)
\dif \nv
\right)^{\frac{1}{p}}
=
\frac{1}{p}
\frac{\dif}{\dif \rho}
\left(
\frac{1}{\rho^k}
\int_{\M} \gamma\Big( \frac{|x|}{\rho}\Big)
\dif \nv
\right)
\left(
\frac{1}{\rho^k}
\int_{\M} \gamma\Big( \frac{|x|}{\rho}\Big)
\dif \nv
\right)^{\frac{1-p}{p}}.
\end{equation}
To do so, we want to bound from below the derivative in the right-hand side to get a differential inequality. We have
\begin{equation*}
\begin{split}
\frac{\dif}{\dif \rho}
\left(
\frac{1}{\rho^k}
\int_{\M} \gamma\Big( \frac{|x|}{\rho}\Big)
\dif \nv(x)
\right)
=
-\frac{1}{\rho^{k+1}}
\int_{G_k(\M)}
\left( k \gamma\Big( \frac{|x|}{\rho}\Big)
+
\frac{|x|}{\rho} \gamma'\Big( \frac{|x|}{\rho}\Big)
\right) \dif V(x,S).
\end{split}
\end{equation*}
Let us choose $X(x)=\gamma\Big( \frac{|x|}{\rho}\Big) x$. Then
\begin{equation*}
\begin{split}
\dive_S X(x)
=
k \gamma\Big( \frac{|x|}{\rho}\Big)
+
\frac{|x|}{\rho} \gamma'\Big( \frac{|x|}{\rho}\Big)
\left| P_S\frac{x}{|x|}\right|^2 .
\end{split}
\end{equation*}
We use Corollary \ref{cor:bv_free_boundary_varifolds}: by testing \eqref{eq:total_fvf_varifold_free_boundary} with $X$ we get the following \emph{monotonicity identity}:
\begin{equation}\label{eq:monotonicity_identity_fb}
\begin{split}
\frac{\dif}{\dif \rho}
\left(
\frac{1}{\rho^k}
\int_{\M} \gamma\Big( \frac{|x|}{\rho}\Big)
\dif \nv(x)
\right)
= &
-
\frac{1}{\rho^{k+1}} \int_{G_k(\M)} \dive_S X(x) \dif V(x,S)
\\
& \qquad -
\frac{1}{\rho^{k+1}} \int_{G_k(\M)} \frac{|x|}{\rho} \gamma'\Big( \frac{|x|}{\rho}\Big) \left| P_{S^\perp}\frac{x}{|x|}\right|^2 \dif V(x,S)
\\
= 
& \frac{1}{\rho^{k+1}} \int_\M \scal{X}{H+\tilde{H}} \dif \nv
- \frac{1}{\rho^{k+1}} \int_{\pM} \scal{X}{N} \dif \muv
\\
& \qquad
-
\frac{1}{\rho^{k+1}} \int_{G_k(\M)} \frac{|x|}{\rho} \gamma'\Big( \frac{|x|}{\rho}\Big) \left| P_{S^\perp}\frac{x}{|x|}\right|^2 \dif V(x,S)
\end{split}
\end{equation}
We have to estimate from below the last member of the above identity.

Since $\gamma' \leq 0$, we can neglect the last integral and, since $|x| \leq \rho$ by cut-off, we obtain
\begin{equation}\label{eq:monotonicity_estimate_fb}
\begin{split}
\frac{\dif}{\dif \rho}
\left(
\frac{1}{\rho^k}
\int_{\M} \gamma\Big( \frac{|x|}{\rho}\Big)
\dif \nv(x)
\right)
\geq &
- \frac{1}{\rho^{k}} \int_\M \gamma\Big( \frac{|x|}{\rho}\Big)|H +\tilde{H}| \dif \nv
\\
& -
\frac{1}{\rho^{k}} \int_{\pM} \gamma\Big( \frac{|x|}{\rho}\Big)
\Bigl| \scal{\frac{x}{|x|}}{N} \Bigr|
\dif \muv.
\end{split}
\end{equation}
We have now to bound the two terms in the right-hand side of \eqref{eq:monotonicity_estimate_fb}.
\begin{itemize}
\item
For the first one, since $H \in L^p(\M,\nv)$ and $\tilde{H} \in L^{\infty}(\pM,\nv)$, also $H + \tilde{H} \in L^p(\M,\nv)$. Therefore, by H\"older inequality we get
\begin{equation}\label{eq:estimate_curvature_lp}
\begin{split}
\frac{1}{\rho^k}
\int_\M \gamma\Big( \frac{|x|}{\rho}\Big)|H+\tilde{H}| \dif \nv
\leq &
\frac{1}{\rho^k}
\left(
\int_{B_\rho} |H +\tilde{H}|^p \dif \nv
\right)^{\frac{1}{p}}
\left(
\int_\M \gamma^{\frac{p}{p-1}}\Big( \frac{|x|}{\rho}\Big) \dif \nv
\right)^{1-\frac{1}{p}}
\\
\leq &
\left(
\frac{1}{\rho^s}
\int_{B_\rho} |H +\tilde{H}|^p \dif \nv
\right)^{\frac{1}{p}}
\rho^{-\frac{k-s}{p}}
\left( \frac{1}{\rho^k}
\int_\M \gamma\Big( \frac{|x|}{\rho}\Big) \dif \nv
\right)^{1-\frac{1}{p}}
\end{split}
\end{equation}
\item
We now move on the estimate of the second integral in the right-hand side of \eqref{eq:monotonicity_estimate_fb}. Since $\pM$ is of class $C^2$ and since $0 \in \pM$, there exists a constant $c$ such that
\begin{equation}\label{eq:estimate_scalar_normal_radial}
|\scal{\frac{x}{|x|}}{N(x)}|
\leq
c |x|.
\end{equation}
This yields
\begin{equation}\label{eq:first_estimate_measure_monotonicity}
\frac{1}{\rho^{k}} \int_{\pM} \gamma\Big( \frac{|x|}{\rho}\Big)
\big|\scal{\frac{x}{|x|}}{N}\big|
\dif \muv
\leq                                  						%prima stima
\frac{c}{\rho^{k-1}} \int \gamma\Big( \frac{|x|}{\rho}\Big) \dif \muv.
\end{equation}
We have to further estimate the righ-hand side of this inequality. This is done by testing \eqref{eq:total_fvf_varifold_free_boundary} with $X(x)=-\gamma\Big( \frac{|x|}{r}\Big) \nabla d(x)$; as in \eqref{eq:estimate_approximate_balls_b_measure} we get
\begin{equation*}
\begin{split}                          						%prima stima
\frac{1}{\rho^{k-1}} \int \gamma\Big( \frac{|x|}{\rho}\Big) \dif \muv
= &																%seconda stima
-\frac{1}{\rho^{k-1}} \int_{G_k(\M)} \gamma'\Big(\frac{|x|}{\rho}\Big) \frac{1}{\rho} \scal{P_S \frac{x}{|x|}}{\nabla d(x)} \dif V(x,S)
\\
& -
\frac{1}{\rho^{k-1}} 
\int_{G_k(\M)} \gamma\Big( \frac{|x|}{\rho}\Big) \dive_S \nabla d(x) \dif V(x,S)
\\
& -
\frac{1}{\rho^{k-1}} 
\int_\M \gamma\Big( \frac{|x|}{\rho}\Big) \scal{\nabla d(x)}{H +\tilde{H}} \dif \nv(x).
\end{split}
\end{equation*}
Since $\gamma'(s)=0$ if $s \in (0,1/2)$, we have $\frac{|x|}{\rho} \geq \frac{1}{2}$; Moreover, using $|\dive_S \nabla d|\leq c$ (because $\rho<R$ and $d$ is of class $C^2$ in $\clos{U_R(\pM)}$ ) and \eqref{eq:estimate_curvature_lp}, we obtain
\begin{equation}\label{eq:penultima_stima_monotonicity}
\begin{split}
\frac{1}{\rho^{k-1}} \int \gamma\Big( \frac{|x|}{\rho}\Big) \dif \muv
\leq &																%terza stima
-\frac{2}{\rho^{k-1}} \int_{\M} \gamma'\Big(\frac{|x|}{\rho}\Big) \frac{|x|}{\rho^2}
\dif \nv(x)
+ \frac{c}{\rho^{k-1}}  \int_\M \gamma\Big( \frac{|x|}{\rho}\Big) \dif \nv(x)
\\
& +
\rho^{1-\frac{k-s}{p}}
\left(
\frac{1}{\rho^s}
\int_{B_\rho} |H +\tilde{H}|^p \dif \nv
\right)^{\frac{1}{p}}
\left( \frac{1}{\rho^k}
\int_\M \gamma\Big( \frac{|x|}{\rho}\Big) \dif \nv
\right)^{1-\frac{1}{p}}
\\
= & 														%quarta stima
\frac{2}{\rho^{k-1}}
\frac{\dif}{\dif \rho}
\left(
\int_{\M} \gamma\Big( \frac{|x|}{\rho}\Big) \dif \nv(x)
\right)
-\frac{2(k-1)}{\rho^k} \int_{\M} \gamma\Big( \frac{|x|}{\rho}\Big) \dif \nv(x)
\\
& +
\frac{2(k-1) +c\rho}{\rho^{k}}  \int_\M \gamma\Big( \frac{|x|}{\rho}\Big) \dif \nv(x)
\\
& +
\rho^{1-\frac{k-s}{p}}
\left(
\frac{1}{\rho^s}
\int_{B_\rho} |H +\tilde{H}|^p \dif \nv
\right)^{\frac{1}{p}}
\left( \frac{1}{\rho^k}
\int_\M \gamma\Big( \frac{|x|}{\rho}\Big) \dif \nv
\right)^{1-\frac{1}{p}}
\end{split}
\end{equation}
Concerning the last member, we now observe that
\begin{equation*}
\begin{split}
\frac{2}{\rho^{k-1}}
\frac{\dif}{\dif \rho}
\left(
\int_{\M} \gamma\Big( \frac{|x|}{\rho}\Big) \dif \nv(x)
\right)
-&
\frac{2(k-1)}{\rho^k} \int_{\M} \gamma\Big( \frac{|x|}{\rho}\Big) \dif \nv(x)
\\
= &
2 \frac{\dif}{\dif \rho}
\left(
\frac{1}{\rho^{k-1}}
\int_{\M} \gamma\Big( \frac{|x|}{\rho}\Big) \dif \nv(x)
\right)
\end{split}
\end{equation*}
Substituting in \eqref{eq:penultima_stima_monotonicity} we get
\begin{equation*}
\begin{split}
\frac{1}{\rho^{k-1}} \int \gamma\Big( \frac{|x|}{\rho}\Big) \dif \muv
\leq &	\,												%quinta stima
\, 2 \frac{\dif}{\dif \rho}
\left(
\frac{1}{\rho^{k-1}}
\int_{\M} \gamma\Big( \frac{|x|}{\rho}\Big) \dif \nv(x)
\right)
\\
& +
\frac{2(k-1) +c\rho}{\rho^{k}}  \int_\M \gamma\Big( \frac{|x|}{\rho}\Big) \dif \nv(x)
\\
& +
\rho^{1-\frac{k-s}{p}}
\left(
\frac{1}{\rho^s}
\int_{B_\rho} |H +\tilde{H}|^p \dif \nv
\right)^{\frac{1}{p}}
\left( \frac{1}{\rho^k}
\int_\M \gamma\Big( \frac{|x|}{\rho}\Big) \dif \nv
\right)^{1-\frac{1}{p}}.
\end{split}
\end{equation*}
Taking into account that
\begin{equation*}
\begin{split}
\frac{\dif}{\dif \rho}
&
\left(
\frac{1}{\rho^{k-1}}
\int_{\M} \gamma\Big( \frac{|x|}{\rho}\Big) \dif \nv(x)
\right)
\\
& \qquad =
\frac{1}{\rho^{k}}
\int_{\M} \gamma\Big( \frac{|x|}{\rho}\Big) \dif \nv(x)
+
\rho
\frac{\dif}{\dif \rho}
\left(
\frac{1}{\rho^{k}}
\int_{\M} \gamma\Big( \frac{|x|}{\rho}\Big) \dif \nv(x)
\right),
\end{split}
\end{equation*}
we obtain
\begin{equation}\label{eq:last_estimate_measure_monotonicity}
\begin{split}
\frac{1}{\rho^{k-1}} \int \gamma\Big( \frac{|x|}{\rho}\Big) \dif \muv
\leq &	\,	
\frac{2k + c\rho}{\rho^{k}}
\int_{\M} \gamma\Big( \frac{|x|}{\rho}\Big) \dif \nv(x)
+
2\rho
\frac{\dif}{\dif \rho}
\left(
\frac{1}{\rho^{k}}
\int_{\M} \gamma\Big( \frac{|x|}{\rho}\Big) \dif \nv(x)
\right)
\\
& +
\rho^{1-\frac{k-s}{p}}
\left(
\frac{1}{\rho^s}
\int_{B_\rho} |H +\tilde{H}|^p \dif \nv
\right)^{\frac{1}{p}}
\left( \frac{1}{\rho^k}
\int_\M \gamma\Big( \frac{|x|}{\rho}\Big) \dif \nv
\right)^{1-\frac{1}{p}}.
\end{split}
\end{equation}
This complete the estimate of the second term in the right-hand side of \eqref{eq:monotonicity_estimate_fb}.
\end{itemize}
Gathering \eqref{eq:estimate_curvature_lp}, \eqref{eq:first_estimate_measure_monotonicity} and \eqref{eq:last_estimate_measure_monotonicity} in \eqref{eq:monotonicity_estimate_fb}, we can estimate \eqref{eq:monotonicity_derivative_initial} as follows:
\begin{equation*}
\begin{split}
\frac{\dif}{\dif \rho}
\left(
\frac{1}{\rho^k}
\int_{\M} \gamma\Big( \frac{|x|}{\rho}\Big)
\dif \nv
\right)^{\frac{1}{p}}
= &
\frac{1}{p}
\frac{\dif}{\dif \rho}
\left(
\frac{1}{\rho^k}
\int_{\M} \gamma\Big( \frac{|x|}{\rho}\Big)
\dif \nv
\right)
\left(
\frac{1}{\rho^k}
\int_{\M} \gamma\Big( \frac{|x|}{\rho}\Big)
\dif \nv
\right)^{\frac{1-p}{p}}
\\
\geq &
- \rho^{-\frac{k-s}{p}}
\Big(
\frac{1}{p} + c\rho
\Big)
\left(
\frac{1}{\rho^s}
\int_{B_\rho} |H +\tilde{H}|^p \dif \nv
\right)^{\frac{1}{p}}
\\
& -
c(1+\rho)
\left( \frac{1}{\rho^k}
\int_{\M} \gamma\Big( \frac{|x|}{\rho}\Big) \dif \nv(x)
\right)^{\frac{1}{p}}
\\
& -
c \rho
\frac{\dif}{\dif \rho}
\left(
\frac{1}{\rho^{k}}
\int_{\M} \gamma\Big( \frac{|x|}{\rho}\Big) \dif \nv(x)
\right)^{\frac{1}{p}}.
\end{split}
\end{equation*}
which is the desired inequality.
\end{proof}

\begin{corollary}\label{cor:monotonicity_formula}
Let $V \in \V_k(\M)$ with free boundary at $\pM$, with $H \in L^p(\M,\nv)$ for some $p \in (k,+\infty)$. Then there exists $\Lambda = \Lambda (k,p,\M,\norm{H}_{L^p})>0$ such that, for all $x_0 \in \pM$ the function
\begin{equation}\label{eq:monotonicity_mass_p}
\rho
\longmapsto
e^{\Lambda \rho}
\left(
\frac{\nv\big( B_\rho(x_0)\big)}{\rho^k}
\right)^{\frac{1}{p}}
+ \Lambda e^{\Lambda \rho} \rho^{1- \frac{k}{p}}
\end{equation}
is monotone increasing.
\end{corollary}
\begin{proof}
Without loss of generality we can assume $x_0=0 \in \pM$. We first notice that we have to prove the result only for small $\rho$, since it is clearly true for $\rho> \diam(\M)$.

We choose $s=0$ in \eqref{eq:monotonicity_inequality_s}; therefore $\frac{k}{p} \in (0,1)$ since $p>k$.  
Rearranging we obtain the existence of $\Lambda > 0$, which depends on the constant $c$ of Lemma \ref{lmm:monotonicity_inequality_s}, on $\norm{H}_{L^p(\M)}$ and on the second fundamental form of $\pM$, such that
\begin{equation*}
\begin{split}
\frac{\dif}{\dif \rho}
e^{\Lambda \rho}
\left[
\left(
\frac{1}{\rho^k}
\int_{\M} \gamma\Big( \frac{|x|}{\rho}\Big)
\dif \nv(x)
\right)^{\frac{1}{p}}
%m(\rho)^{\frac{1}{p}}
+ \Lambda \rho^{1- \frac{k}{p}}
\right]
\geq
0.
\end{split}
\end{equation*}
Since the constant $c$ of Lemma \ref{lmm:monotonicity_inequality_s} does not depend on the choice of $\gamma$ (the estimates are indipendent on the choice of $\gamma$ unless that $\gamma'(s)=0$ for $s \in (0,1/2)$), letting $\gamma$ increase to $\ind_{[0,1)}$ we have that the function
\begin{equation*}
\rho
\longmapsto
e^{\Lambda \rho}
\left[
\left(
\frac{\nv(B_\rho)}{\rho^k}
\right)^{\frac{1}{p}}
+ \Lambda \rho^{1- \frac{k}{p}}
\right]
\end{equation*}
is monotone increasing.
\end{proof}
If $H \in L^{\infty}(\M, \nv)$, in \eqref{eq:estimate_curvature_lp} we simply estimate
\begin{equation*}
\begin{split}
\frac{1}{\rho^k}
\int_\M \gamma\Big( \frac{|x|}{\rho}\Big)|H +\tilde{H}| \dif \nv
\leq &
\frac{\normi{H + \tilde{H}}}{\rho^k}
\int_\M \gamma \Big( \frac{|x|}{\rho}\Big) \dif \nv
\end{split}
\end{equation*}
By repeating the previous arguments we have the following result:
\begin{corollary}\label{cor:monotonicity_formula_infty}
Let $V \in \V_k(\M)$ have free boundary at $\pM$ with $H \in L^\infty(\M,\nv)$. Then there exists $\Lambda = \Lambda (k,\M,\normi{H})$ such that, for all $x \in \pM$ the function
\begin{equation}\label{eq:monotonicity_mass_infty}
\rho
\longmapsto
e^{\Lambda \rho}
\frac{\nv\big( B_\rho(x)\big)}{\rho^k}
\end{equation}
is monotone increasing.
\end{corollary}

\subsection{Density set and proof of Theorem \ref{thm:singular_set_V_lower_dimensional}}\label{sec:proof_regular_set}
In order to prove Theorem \ref{thm:rectifiability_free_boundary} when $p \in (1,k]$, we have to show that, for every varifold $V$ with free boundary at $\pM$ with $H \in L^p(\M,\nv)$, the Hausdorff dimension of the set of points where the $k$-density $\T^k(\nv,\cdot)$ does not exist or is infinite is at most $k-p$.

We introduce the definition of \emph{density set} for a $k$-varifold $V$.

\begin{defi}[Density set]\label{def:regular_points_varifolds}
Let $V \in V_k(\M)$ has free boundary at $\pM$ and $H \in L^p(\M,\nv)$ for some $p \geq 1$. Then the density set for $V$ is defined as
\begin{gather*}
\reg(V)
= \bigcup_{s \in (k-p, k]}
\Big\{
x \in \M
\mid
\limsup_{r \to 0}
\frac{1}{r^{s}} \int_{B_{r}(x)} |H + \tilde{H}|^p \dif \norm{V}
= 0
\Big\}.
%\notag
\end{gather*}
\end{defi}
\begin{rem}
We notice that, if $p>k$, then $\reg(V)= \M$ ($s=0 \in (k-p,k]$) and Theorem \ref{thm:singular_set_V_lower_dimensional} is an easy consequence of the monotonicity of density ratios, Corollaries \ref{cor:monotonicity_formula} and \ref{cor:monotonicity_formula_infty}. Thus Theorem \ref{thm:singular_set_V_lower_dimensional} is an extension to lower integrability of the mean curvature of the well-known existence of the density at every point when $p > k$.
\end{rem}

\begin{proof}[Proof of Theorem \ref{thm:singular_set_V_lower_dimensional}]
We first show that $\haus{s}\big(M \- \reg(V) \big)=0$ for all $s> k-p$. We have
\begin{equation*}
\reg(V) = \bigcup_{s \in (k-p, k]} A_s;
\qquad
A_s = \Big\{
x \in \M
\mid
\limsup_{r \to 0}
\frac{1}{r^{s}} \int_{B_{r}(x)} |H + \tilde{H}|^p \dif \norm{V}
= 0
\Big\}.
\end{equation*}
If $\mu$ is absolutely continuous with respect to $\nv$ and has density $|H+\tilde{H}|^p$, that is $\mu=|H+ \tilde{H}|^p\nv$, then for any $s \in (k-p, k]$ we have
\begin{equation*}
\M \- A_s = \bigcup_{i \in \N}\Big\{x \in \M \mid \T^{*s}(\mu,x) \geq \frac{1}{i}\Big\}.
\end{equation*}
Then, by \cite[Theorem 6.9]{mattila_1995}, for every $i \in \N$
\begin{equation*}
\haus{s}\Big( \big\{x\mid \T^{*s}(\mu,x) \geq \frac{1}{i}\big\}\Big)
\leq i
\mu\Big( \big\{x\mid \T^{*s}(\mu,x) \geq \frac{1}{i}\big\}\Big).
\end{equation*}
We want to show that for every $i \in \N$ the right-hand side is equal to $0$, which would prove \eqref{eq:dimension_singular_points}. To see this, we first recall that for a varifold with bounded first variation in $\M$ the $\tk(\nv,x)$ exists and is finite for $\nv$-a.e. $x \in \M$ by \cite[Lemma 40.5]{simon:lectures}; since 
\begin{equation*}
\lim_{r \to 0} \frac{1}{\nv(B_r(x))} \int_{B_r(x)} |H+\tilde{H}|^p \dif \nv
< \infty
\qquad
\text{for } \nv\text{-a.e. } x \in \M,
\end{equation*}
we have
\begin{equation*}
\mu\Big( \big\{x\mid \T^{*s}(\mu,x) \geq \frac{1}{i}\big\}\Big)
=
0.
\end{equation*}
Thus
$\haus{s}\big( \M \- \reg(V) \big) \leq \haus{s}\big( \M \- A_s \big) = 0$. This shows \eqref{eq:dimension_singular_points}.

In order to prove \eqref{eq:monotonicity_phi}, we fix $x_0 \in \reg(V) \cap \pM$ and $s \in (k-p,k]$ such that $x_0 \in A_s$. Without loss of generality we can assume $x_0=0$. It is enough to show the result only for small radii, because it is clearly valid if $\diam(\M) \leq r < t$. By the choice of $s$, we have $\frac{k-s}{p} \in [0,1)$. Applying Lemma \ref{lmm:monotonicity_inequality_s} and taking into account $0 \in A_s$, we obtain that exists $\Lambda>0$ (which depends on the point $x_0$ chosen) such that
\begin{equation*}
\begin{split}
\frac{\dif}{\dif \rho}
e^{\Lambda \rho}
\left[
\left(
\frac{1}{\rho^k}
\int_{\M} \gamma\Big( \frac{|x|}{\rho}\Big)
\dif \nv(x)
\right)^{\frac{1}{p}}
%m(\rho)^{\frac{1}{p}}
+ \Lambda \rho^{1- \frac{k-s}{p}}
\right]
\geq
0.
\end{split}
\end{equation*}
Since this is independent on the choice of $\gamma$, letting $\gamma$ to increase to $\ind_{[0,1)}$ we obtain \eqref{eq:monotonicity_phi}.

If $x \in \reg(V) \cap \interior{M}$ the proof of \eqref{eq:monotonicity_phi} is the same, except for the fact that we have to use the classical monotonicity inequality for interior points, which indeed is
\begin{equation*}\label{eq:monotonicity_inequality_s_interior}
\begin{split}
\frac{\dif}{\dif \rho}
\left(
\frac{1}{\rho^k}
\int_{\M} \gamma\Big( \frac{|x-x_0|}{\rho}\Big)
\dif \nv
\right)^{\frac{1}{p}}
\geq &
- \frac{\rho^{-\frac{k-s}{p}}}
{p}
\left(
\frac{1}{\rho^s}
\int_{B_\rho} |H|^p \dif \nv
\right)^{\frac{1}{p}}
\\
& -
\left( \frac{1}{\rho^k}
\int_{\M} \gamma\Big( \frac{|x-x_0|}{\rho}\Big) \dif \nv(x)
\right)^{\frac{1}{p}}.
\end{split}
\end{equation*}

The existence of $\T^k(\nv,\cdot)$ on $\reg(V)$ and the upper semi-continuity of the restrictions of $\tkv{\cdot}$ to $\reg(V) \cap \interior{\M}$ and to $\reg(V) \cap \pM$ are easy consequences of \eqref{eq:monotonicity_phi}.
\end{proof}

\section{Proof of Theorem \ref{thm:rectifiability_free_boundary}}\label{sec:proof_rectifiability}
In this section we often use the following assumption on $V$.
\begin{assumption}\label{ass:varifold_rectifiability}
$V \in \V_k(\M)$ is a rectifiable $k$-varifold with free boundary at $\pM$, with generalized mean curvature $H \in L^p(\M,\nv)$ for some $p\geq 1$ and $\tkv{x}\geq 1$ for $\nv$-a.e.\ $x \in \M$.
\end{assumption}
%$V \in \V_k(\M)$ is a rectifiable $k$-varifold with free boundary at $\pM$, with generalized mean curvature $H \in L^p(\M,\nv)$ for some $p>1$ and $\tkv{x}\geq 1$ for $\nv$-a.e.\ $x \in \M$.
We remark that if $V$ satisfies Assumption \ref{ass:varifold_rectifiability}, Corollary \ref{cor:bv_free_boundary_varifolds} establishes the existence of the measure $\muv$ and that $V$ has bounded first variation.

We first recall the definition of $\mus$, the \emph{$(k-1)$-dimensional part} of $\muv$:
\begin{equation*}
\mus
=
{\muv}
\llcorner
E,
\qquad
E=
\{
x \mid
0< \tlu(\muv,x) \leq \tuu(\muv,x) < +\infty
\}.
\end{equation*}
As stated in the introduction, we recall that if $H \in L^p(\M, \nv)$ for some $p>k$, then the condition $\tuu(\muv,x) < +\infty$ is not restrictive, since it holds \emph{for every} point $x \in \pM$; indeed  \eqref{eq:estimate_measure_boundary} and H\"older inequality yield
\begin{equation*}
\begin{split}
\frac{\muv(B_{r/2}(x)}{r^{k-1}}
\leq &
c \frac {\nv \big( B_r(x) \big)}{r^k}
+
\frac{1}{r^{k-1}}\int_{B_r(x)} |H| \dif \nv
\\
\leq &
c \frac {\nv \big( B_r(x) \big)}{r^k}
+
r^{1- \frac{k}{p}}\bigg(\int_{B_r(x)} |H|^p \dif \nv \bigg)^{\frac{1}{p}}
\bigg( \frac {\nv \big( B_r(x) \big)}{r^k} \bigg)^{1-\frac{1}{p}};
\end{split}
\end{equation*}
By Lemma \ref{cor:monotonicity_formula} and $p>k$, the last member of the above inequality is bounded as $r \to 0$, therefore it follows $\tuu(\muv,x) < +\infty$.
This proves that Theorem \ref{thm:rectif_fb_p>k} is a corollary of Theorem \ref{thm:rectifiability_free_boundary}.
 
Since in the proof of Theorem \ref{thm:rectifiability_free_boundary} we have to deal with the blow-ups of $V$ at points on $\pM$, we introduce a notation for the scalings of $V$ at a point $x_0 \in \pM$ that we use throughout this section.
\begin{notation}[Scalings]\label{not:scalings}
If $V \in \V_k(\M)$ has free boundary at $\pM$, if $x_0 \in \pM$ is a fixed point and $r_j \downarrow 0$ is a fixed sequence, we use the following notations:
\begin{equation*}
\M_j:= \dil{x_0,r_j}(\M)
\qquad
V_j := (\dil{x_0,r_j})_\sharp V
\in \V_k(\M_j)
\qquad
\muj := \muvj;
\end{equation*}
where $\dil{x_0,r_j}$ is the dilation function defined in \eqref{eq:dilation_map} and $f_\sharp V$ denotes the push-forward of $V$ through $f$ defined in \eqref{eq:push-forward_varifolds}.
In Lemma \ref{lmm:tangent_cones} we show that each $V_j$ has free boundary at $\pM_j$, thus Corollary \ref{cor:bv_free_boundary_varifolds} states the existence of $\muvj$, which we call $\muj$. Moreover, we denote by $H_j$ the generalized mean curvature of $V_j$ and $\tilde{H}_j$ the vector field provided by Corollary \ref{cor:bv_free_boundary_varifolds} relative to $V_j$. If $x \in \pM_j$, we denote by $N_j(x)$ the exterior unit normal to $\M_j$ at $x$. As $j \to \infty$, according to the definition of tangent space given in section \ref{sec:notations}, we say that $\M_j \to T_{x_0}^+\M$ and that $\pM_j \to T_{x_0}\pM$.
\end{notation}

The idea of the proof is to study the blow-ups of $V$ at points on $\pM$, to prove that the at $\mus$-a.e.\ point on $\pM$, every $(k-1)$-blow-up of $\mus$ is of the form $\alpha \haus{k-1}\llcorner S$ for some $(k-1)$-dimensional linear subspace $S$. This allows us to apply the Marstrand-Mattila Rectifiability Criterion \cite[Theorem 5.1]{delellis:rectifiable} to $\mus$. We state explicitly the criterion for the reader convenience:
\begin{theorem}[Marstrand-Mattila Rectifiability Criterion]\label{thm:M-M_criterion}
Let $m \leq n$ be a natural number and let $\mu$ be a positive Radon measure on $\R^n$ such that, for $\mu$-a.e.\ $x \in \R^n$, we have
\begin{enumerate}
\item $0 < \T_*^m(\mu,x) \leq \T^{* m}(\mu,x) < \infty$;
\item Every $m$-tangent measure of $\mu$ at $x$ is of the form $\beta \haus{m} \llcorner S$ for some $m$-dimensional linear subspace of $\R^n$.
\end{enumerate}
Then $\mu$ is $m$-rectifiable.
\end{theorem}
We begin in subsection $\ref{subsec:well_known_facts}$ where we adapt to $\mus$ two well-known facts about measures:
\begin{itemize}
\item in Lemma \ref{lmm:mus_abscont_hausk-1} we show that $\mus\big( \M \- \reg(V) \big)=0$; this allows us to check the conditions of the Marstrand-Mattila criterion just on $\reg(V)$.
\item In Lemma \ref{lmm:tangent_mus_equal_muv} we prove that for $\mus$-a.e.\ $x \in \pM$, $\muv$ and $\mus$ have the same $(k-1)$-tangent measures. 
\end{itemize}

Next, since by definition
\begin{equation*}
0< \tlu(\mus,x) \leq \tuu(\mus,x) < +\infty
\qquad
\mbox{for } \mus \mbox{-a.e. } x \in \pM,
\end{equation*}
to apply Marstrand-Mattila Rectifiability Criterion it remains to show that for $\mus$-a.e.\ $x_0 \in \reg(V)$, every tangent measure to $\muv$ is a $(k-1)$-dimensional plane. Subsection \ref{subsec:proof_second_condition_mm} is devoted to prove this, which is achieved in several steps.
The outline of the proof is the following:
\begin{enumerate}
\item In Lemma \ref{lmm:tangent_cones} we first prove that at every $x_0 \in \reg(V) \cap \pM$, tangent varifolds to $V$ are cones in the half-space $T^+_{x_0}\M$ that are stationary with respect to $\X_t(T^+_{x_0}\M)$.
Thus, if $C\in \tang(V,x_0)$ for some $x_0 \in \reg(V) \cap \pM$, Corollary \ref{cor:bv_free_boundary_varifolds} provides the existence of the measure $\muc$. We also prove that the measures $\muj$
relative to the scalings of $V$ converge weakly to $\muc$; This imply that (Corollary \ref{cor:tangent_mus_are_muc})
\begin{equation}\label{eq:tangent_b_measures_are_boundary}
\tang^{k-1}(\mus,x_0)
=
\{
\muc \mid C \in \tang(V,x_0)
\}
\qquad
\mbox{for } \mus\mbox{-a.e. } x_0 \in \pM.
\end{equation}

\item In Lemma \ref{lmm:max_density_invariance} we prove that each tangent cone $C$ has an invariant linear subspace $D_C$ that coincides with the set of points where the $k$-density $\T^k(\norm{C},\cdot)$ attains its maximum;
\item In Lemma \ref{lmm:approx_continuity_contant_density} we show that for $\mus$-a.e.\ $x_0 \in \pM$, if $C \in \tang(V,x_0)$, then $\muc$ is concentrated on $D_C$;
\item In Lemma \ref{lmm:positive_density_k-1-dimensional} we prove that for $\mus$-a.e.\ $x_0 \in \pM$, if $C \in \tang(V,x_0)$, then $D_C$ is $(k-1)$-dimensional and $\muc = \alpha \haus{k-1}\llcorner D_C$ for some constant $\alpha=\alpha(x_0,C)$;
\item
In subsection \ref{subsec:final} we summarize all these facts to conclude the proof of Theorem \ref{thm:rectifiability_free_boundary}: by Lemma \ref{lmm:positive_density_k-1-dimensional} and \eqref{eq:tangent_b_measures_are_boundary} it follows that, for $\mus$-a.e.\ $x_0 \in \pM$, every $(k-1)$-tangent measure to $\muv$ at $x_0$ is a $(k-1)$-plane.
\end{enumerate}

\subsection{Summary of well-known facts}\label{subsec:well_known_facts}
We begin with a measure-theoretic lemma applied to $\mus$.
\begin{lemma}\label{lmm:mus_abscont_hausk-1}
Let $V \in \V_k(\M)$ has free boundary at $\pM$ with $H \in L^p(\M, \nv)$ for some $p\geq 1$. Then $\mus \ll \haus{k-1}$.
In particular, if $p>1$, we have
\begin{equation}\label{eq:singular_points_mus-negligible}
\mus\big(\pM \- \reg(V) \big)=0
\end{equation}
\end{lemma}
\begin{proof}
The proof is an easy consequence of \cite[Theorem 6.9]{mattila_1995} which states that if $\mu \in \mep(\R^n)$ is a positive Radon measure and if
\begin{equation*}
A \subset
\{
x \in \R^n \mid \tuu(\mu,x) \leq \lambda
\},
\end{equation*}
then $\mu(A) \leq 2^{k-1}\lambda \haus{k-1}(A)$. Let us assume $\haus{k-1}(A)=0$; since by definition $\tuu(\mus,x)< \infty$ for $\mus$-a.e.\ $x \in \M$ we have that
\begin{equation*}
\begin{split}
\mus(A)
=
\mus \Big(
\bigcup_{i \in \N} A \cap
\{
\tuu(\mus,x)< i
\}
\Big)
\leq
\sum_{i \in \N} i 
2^{k-1}\haus{k-1}(A)
=
0.
\end{split}
\end{equation*}
\eqref{eq:singular_points_mus-negligible} clearly follows by $\mus \ll \haus{k-1}$ and \eqref{eq:dimension_singular_points}.
\end{proof}

Since $\mus$ is concentrated on $E$, we expect that $\tang^{k-1}(\mus,x) = \tang^{k-1}(\muv,x)$ for every point of $E$ that has density $1$ with respect to $\muv$. This is a well-known fact, see e.g. \cite[Remark 3.13]{delellis:rectifiable}, but we recall the proof for the reader convenience.
\begin{lemma}\label{lmm:tangent_mus_equal_muv}
Let $V \in \V_k(\M)$ has free boundary at $\pM$ with $H \in L^p(\M, \nv)$ for some $p\geq 1$, let $\muv$ be the measure provided by Corollary \ref{cor:bv_free_boundary_varifolds} and let $E$ be defined as in \eqref{eq:definition_mus}. Then
\begin{equation}\label{eq:tangent_mus_equal_muv}
\tang^{k-1}(\mus,x_0)
=
\tang^{k-1}(\muv,x_0)
\qquad
%\mbox{for} \mus-\mbox{a.e. } x_0 \in \pM.
\forall x_0 \in E \text{ such that }
\lim_{r \to 0} \frac{\muv \big( E \cap B_r(x) \big)}{\muv \big(B_r(x) \big)}
=
%\lim_{r \to 0} \frac{\mus \big(B_r(x) \big)}{\muv \big(B_r(x) \big)}
%=
1.
\end{equation}
In particular, the first equality in \eqref{eq:tangent_mus_equal_muv} holds $\mus$-a.e..
\end{lemma}
\begin{proof}
Let us fix $x_0 \in \pM$ such that
\begin{equation*}
\lim_{r \to 0} \frac{\muv \big( E \cap B_r(x) \big)}{\muv \big(B_r(x) \big)}
=
\lim_{r \to 0} \frac{\mus \big(B_r(x) \big)}{\muv \big(B_r(x) \big)}
=
1,
\end{equation*}
let $\nu \in \tang^{k-1}(\muv,x_0)$ and $r_j \downarrow 0$ such that
\begin{equation}\label{eq:convergence_tangent_nu}
\muj
:=
\frac{1}{r_j^{k-1}} \big( \dil{x_0,r_j}\big)_\# \muv
\wto
\nu.
\end{equation}
We are going to prove that
\begin{equation}\label{eq:scalings_mus_converge_nu}
\muj^*
:=
\frac{1}{r_j^{k-1}} \big( \dil{x_0,r_j}\big)_\# \mus
\wto
\nu.
\end{equation}
Without loss of generality we can assume $x_0=0$. Let us consider $f \in C_c(\R^n)$. We have
\begin{equation*}
\begin{split}
\int f(x) \dif \nu(x)
=&
\lim_{j \to \infty}
\frac{1}{r_j^{k-1}}
\int f \Big( \frac{x}{r_j}\Big) \dif \muv(x)
\\
=&
\lim_{j \to \infty}
\frac{1}{r_j^{k-1}}
\left(
\int_E f \Big( \frac{x}{r_j}\Big) \dif \muv(x)
+
\int_{\R^n \- E} f \Big( \frac{x}{r_j}\Big) \dif \muv(x)
\right).
\end{split}
\end{equation*}
For the last term there exists $c>0$ such that
\begin{equation*}
\begin{split}
\frac{1}{r_j^{k-1}}
\int_{\R^n \- E} \Big| f \Big( \frac{x}{r_j}\Big) \Big| \dif \muv(x)
& \leq
\frac{c}{r_j^{k-1}} \muv\big( B_{cr_j} \- E \big)
\\
& \leq
c \frac{\muv \big( B_{cr_j} \big)}{r_j^{k-1}}
\frac{\muv\big( B_{cr_j} \- E \big)}{\muv \big( B_{cr_j} \big)}
\xrightarrow[j \to \infty]{}
0
\end{split}
\end{equation*}
since $E$ has density $1$ with respect to $\muv$ and $\muv( B_{cr_j})/r_j^{k-1}$ remains bounded by $\muj \wtoi \nu$. Hence
\begin{equation*}
\begin{split}
\int f(x) \dif \nu(x)
=&
\lim_{j \to \infty}
\frac{1}{r_j^{k-1}}
\int f \Big( \frac{x}{r_j}\Big) \dif \mus(x),
\end{split}
\end{equation*}
thus
\begin{equation*}
\muj^*
=
\frac{1}{r_j^{k-1}} \big( \dil{0,r_j}\big)_\# \mus
\wto
\nu.
\end{equation*}
This proves $\tang^{k-1}(\muv,x_0) \subset \tang^{k-1}(\mus,x_0)$; the reverse inclusion can be proved in a similar way.
\end{proof}

\subsection{Proof of the second condition of the Marstrand-Mattila Criterion}\label{subsec:proof_second_condition_mm}
We begin by studying tangent varifolds to $V$ in a point $x \in \pM$.

\begin{lemma}\label{lmm:tangent_cones}
Let $V \in \V_k(\M)$ has free boundary at $\pM$ with $H \in L^p(\M, \nv)$ for some $p\geq1$. Then,
for every $x_0 \in \reg(V) \cap \pM$
the following statements hold: let $r_j \downarrow 0$ be fixed and let us use the notations for the scalings defined on p. \pageref{not:scalings}; then
\begin{enumerate}
\item every $V_j$ has free boundary at  $\pM_j$ with $\tilde{H}_j(x)=r_j \tilde{H}(r_j x)$ for every $x \in \pM_j$ and
\begin{equation}\label{eq:scaled_mu_is_muj}
\muvj
=
\frac{1}{r_j^{k-1}}(\dil{0,r_j})_\# \muv;
\end{equation}

\item there exist a subsequence of $r_j$, not relabeled, and a $k$-varifold $C$ with $\supp C \subset \W$ such that
\begin{equation*}
V_j \wtoi_{j \to \infty} C.
\end{equation*}
In particular, $\tang(V,x_0) \neq \emptyset$.
\item $C$ is stationary with respect to $\X_t(T_{x_0}^+ \M)$;

\item if $\muc$ is the measure given by Corollary \ref{cor:bv_free_boundary_varifolds} relative to $C$, we have
\begin{equation}\label{eq:scalings_muv_to_muc}
\muj
:= \muvj
%=
%\frac{1}{r_j^{k-1}}(\dil{x_0,r_j})_\# \muv
\wtoi_{j \to \infty}
\muc;
\end{equation}

\item If in addition $\tkv{x} \geq 1$ for $\nv$-a.e.\ $x \in \M$ (that is if $V$ satisfies Assumption \ref{ass:varifold_rectifiability}), then $C$ is a rectifiable cone and $\muc$ is scaling invariant, that is
\begin{equation*}
\frac{1}{r^{k-1}}(\dil{0,r})_\# \muc
=
\muc
\qquad
\forall r>0.
\end{equation*}
\end{enumerate}
\end{lemma}

\begin{proof}
To simplify the notations, we assume without loss of generality that $x_0=0 \in \pM \cap \reg(V)$ and $T_0^+\M=\{x_n \geq 0\}$.
\begin{steps}[wide,%
%labelwidth=!,%
labelindent=5pt]
\item
We first prove that every $V_j$ has free boundary at $\pM_j$. For every $X \in \X_t(\M_j)$ the vector field $x \mapsto X(\frac{x}{r_j})$ belongs to $\Xt$, thus we can apply \eqref{eq:fvf_free_boundary_intro} to $V$ and get
\begin{equation*}
\begin{split}
\int_{G_k(\M_j)} \dive_S X(x) \dif V_j(x,S)
= &
\frac{1}{r_j^{k-1}} \int_{G_k(\M)} \dive_S X\big(\frac{x}{r_j}\big) \dif V(x,S)
\\
= &
- \frac{1}{r_j^{k-1}}\int_\M \scal{X\big(\frac{x}{r_j}\big)}{H(x)} \dif \nv(x)
.
\end{split}
\end{equation*}
Thus, if we define $H_j(x)=r_j H(r_j x)$ for all $x \in \M_j$, by changing again variables we get
\begin{equation*}
\int_{G_k(\M_j)} \dive_S X(x) \dif V_j(x,S)
= -
\int_{\M_j} \scal{X(x)}{H_j(x)} \dif \norm{V_j}(x),
\end{equation*}
that is $V_j$ has free boundary at $\pM_j$, with $H_j(x)=r_j H(r_j x)$. By Corollary \ref{cor:bv_free_boundary_varifolds}, $V_j$ has bounded first variation and there exists $\muj= \muvj$ and $\tilde{H_j}$. By a similar argument and by definition of $\tilde{H}$ it easily seen that
\begin{equation*}
%H_j(x) +
\tilde{H_j}(x)
=
r_j\tilde{H}(r_j x)
\qquad
\mbox{ for }
\norm{V_j}
\mbox{-a.e. }
x \in \pM_j.
\end{equation*}
If $X \in \X(\M_j)$, then
\begin{equation*}
\begin{split}
\int_{G_k(\M_j)} \dive_S X(x) \dif V_j(x,S)
= &
\frac{1}{r_j^{k-1}} \int_{G_k(\M)} \dive_S X\big(\frac{x}{r_j}\big) \dif V(x,S)
\\
= &
- \frac{1}{r_j^{k-1}}\int_\M \scal{X\big(\frac{x}{r_j}\big)}{H(x) +\tilde{H}(x)} \dif \nv(x)
\\
& +
\frac{1}{r_j^{k-1}} \int_{\pM} \scal{X\big(\frac{x}{r_j}\big)}{N(x)} \dif \muv(x)
\\
= &
- \int_{\M_j} \scal{X}{H_j +\tilde{H}_j} \dif \norm{V_j}
+
\frac{1}{r_j^{k-1}} \int_{\pM_j} \scal{X}{N_j} \dif (\dil{0,r_j})_\#\muv
\end{split}
\end{equation*}
If we compare this equality with \eqref{eq:total_fvf_varifold_free_boundary} for $V_j$ tested with $X$, we get
\begin{equation*}
\muj
:=
\sigma_{V_j}
=
\frac{1}{r_j^{k-1}}(\dil{0,r_j})_\# \muv,
\end{equation*}
that is $\muj$ is obtained by scaling $\muv$.
\item
We now want to study the limit of the sequence $V_j$.
Since $0 \in \reg(V)$, by Theorem \ref{thm:singular_set_V_lower_dimensional} there exists a constant $c>0$ such that
\begin{equation}\label{eq:masses_scalings_unif_bounded}
\norm{V_j}(B_1)
=
\norm{(\dil{0,r_j})_\sharp V} \big( B_1)
=
\frac{\norm{V}\big( B_{r_j} \big)}{r_j^k}
\leq
c
\qquad
\forall j \in \N.
\end{equation}
Thus by compactness of Radon measures, there exist $C \in \V_k(\R^n)$ and a subsequence of $r_j$, not relabeled, such that
\begin{equation}\label{eq:Vj_to_C}
V_j
\wtoi_{j \to \infty}
C.
\end{equation}
\item
Clearly $\supp C \subset \W$. To show that $C$ is stationary with respect to $\X_t(\W)$, we test with a vector field $X \in \X_t(\W)$. We first need the following estimate on $H_j$ (exactly the same relation holds for $\tilde{H}_j$):
\begin{equation}\label{eq:estimate_curvature_blow-up}
\norm{H_j}_{L^p(B_1)}
=
\left( \int_{B_1} |H_j|^p \dif \norm{V_j} \right)^{\frac{1}{p}}
=
\left( \frac{1}{r_j^{k-p}} \int_{B_{r_j}} |H|^p \dif \norm{V} \right)^{\frac{1}{p}}
%=
%r_j^{1- \frac{k}{p}} \norm{H}_{L^p(B_{r_j})}
\xrightarrow[j \to \infty]{}
0.
\end{equation}
where in the second equality we have used $H_j(x)=r_j H(r_j(x))$ and the limit follows by $0 \in \reg(V)$ and $|H| \leq |H + \tilde{H}|$ because $H(x) \in T_x \pM$ (Remark \ref{rem:fb_curvature_tangent}) while $\tilde{H} \in (T_x \pM)^\perp$.

To prove that $C$ is stationary, let us pick $X \in \X_t(\W)$ with compact support; there exists a sequence of vector fields $X_j \in \X_t(\M_j)$ with compact support such that $X_j \to X$ in the $C^1$ topology. Hence, using \eqref{eq:Vj_to_C} and \eqref{eq:estimate_curvature_blow-up} we get
\begin{equation*}
\begin{split}
\int_{G_k(\W)} \dive_S X(x) \dif C(x,S)
= &
\lim_{j \to \infty}
\int_{G_k(\M_j)} \dive_S X_j(x) \dif V_j(x,S)
\\
%= &
%\int_{G_k(\R^n)} \dive_S X(x) \dif C(x,S)
%\\
%= &
%\lim_{j \to \infty}
%\int_{G_k(\R^n)} \dive_S X(x) \dif V_j(x,S)
%\\
= &
- \lim_{j \to \infty} \int_{\M_j} \scal{X_j(x)}{H_j(x)} \dif \norm{V_j}(x)
\\
= &
\, 0
%- \lim_{j \to \infty} \int_{\M_j} \scal{X(x)}{\widehat{H_j}(x)} \dif \norm{V_j}(x)
%\\
%& +
%\lim_{j \to \infty} \int_{\pM_j} \scal{X(x)}{N_j(x)} \dif \muj(x)
%\\
%= &
%\lim_{j \to \infty} \int_{\pM_j} \scal{X(x)}{N_j(x)} \dif \muj(x)
%%- \int_{\partial W} \scal{X(x)}{e_n} \dif \muc(x),
%%\qquad
%%\forall X \in \Xm,
\end{split}
\end{equation*}
that is, by definition, $C$ is stationary with respect to $\X_t(\W)$.
\item
Corollary \ref{cor:bv_free_boundary_varifolds} provides that $C$ has bounded first variation and the existence of $\muc$. We want to prove that
\begin{equation}\label{eq:bmeasures_convergence_blow-up}
\muj \wtoi_{j \to \infty} \muc.
\end{equation}
To this aim, we first remark that $\supp \muc  \subset T_0 \pM$ and, since 
$T_0 \pM$ is flat, we have that $\tilde{H}_C=0$.
We next need an uniform bound on $\muj(B_1)$ and we use \eqref{eq:estimate_measure_boundary}: since the second fundamental forms of $\pM_j$ go to $0$ as $r_j \to 0$, when we apply \eqref{eq:estimate_measure_boundary} to $V_j$ in $B_1$, the constant $c$ in \eqref{eq:estimate_measure_boundary} is bounded uniformly in $j$. This implies that there exists an uniform constant $c$ such that for each $j \in \N$
\begin{equation}\label{eq:uniform_bound_mu_blow-up}
\begin{split}
\muj (B_1)
= &
\frac{\muv(B_{r_j})}{r_j^{k-1}}
\\
\leq &
\frac{c \nv(B_{2r_j})}{r_j^k}
+
\frac{1}{r_j^{k-1}}\int_{B_{2r_j}}|H| \dif \nv.
\\
\leq &
\frac{c \nv(B_{2r_j})}{r_j^k}
+
r_j^{1 - \frac{k-s}{p}}\bigg( \frac{1}{r_j^s}\int_{B_{2r_j}}|H|^p \dif \nv \bigg)^{\frac{1}{p}}
\bigg(\frac{\nv(B_{2r_j})}{r_j^k} \bigg)^{1- \frac{1}{p}}
\end{split}
\end{equation}
where $s \in (k-p,k]$ is such that $0$ satisfies the $s$-density condition in the definition of $\reg(V)$.
The last member of \eqref{eq:uniform_bound_mu_blow-up} is uniformly bounded in $j$ since $0 \in \reg(V)$, Theorem \ref{thm:singular_set_V_lower_dimensional} and by $\frac{k-s}{p} \leq 1$. This proves the uniform bound on $\muj(B_1)$.

To complete the proof of \eqref{eq:bmeasures_convergence_blow-up}, let $X \in \X(\R^n)$ be a vector field with compact support. If $e_n$ is the $n$-th coordinate unit vector (that is the interior unit normal vector to $\partial \W$), we have
\begin{equation}\label{eq:limit_test_b_measures}
\begin{split}
- \int_{T_0 \partial \M} \scal{X}{e_n} \dif \muc
= &
\int_{G_k(\W)} \dive_S X(x) \dif C(x,S)
\\
%= &
%\int_{G_k(\R^n)} \dive_S X(x) \dif C(x,S)
%\\
%= &
%\lim_{j \to \infty}
%\int_{G_k(\R^n)} \dive_S X(x) \dif V_j(x,S)
%\\
= &
\lim_{j \to \infty}
\int_{G_k(\M_j)} \dive_S X(x) \dif V_j(x,S)
\\
= &
- \lim_{j \to \infty} \left( \int_{\M_j} \scal{X(x)}{H_j(x) + \tilde{H_j}(x)} \dif \norm{V_j}(x)
+
%\lim_{j \to \infty}
\int_{\pM_j} \scal{X(x)}{N_j(x)} \dif \muj(x)
\right)
\\
= &
\lim_{j \to \infty} \int_{\pM_j} \scal{X(x)}{N_j(x)} \dif \muj(x),
\end{split}
\end{equation}
where the first identity follows by $H_C=\tilde{H}_C=0$ and the last one follows by \eqref{eq:estimate_curvature_blow-up}.
Thus $N_j \muj \wtoi -e_n \muc$ is proved, since $\X(\R^n)$ is dense in $C_c(\R^n, \R^n)$, by \eqref{eq:limit_test_b_measures} and the uniform bound on $\muj(B_1)$ \eqref{eq:uniform_bound_mu_blow-up}.
To prove \eqref{eq:bmeasures_convergence_blow-up}, it is enough to observe that, if $f \in C_c(\R^n, \R)$, then
\begin{equation*}
\int f \dif \muc
=
\int \scal{f e_n}{e_n}  \dif \muc
=
- \lim_{j \to \infty} \int \scal{f e_n}{N_j} \dif \muj
=
\lim_{j \to \infty} \int f \dif \muj,
\end{equation*}
by the uniform bound on $\muj(B_1)$ and since $\pM_j \to \partial \W$ in the $C^1$ topology.
\item
We are left to prove that, if  $\tkv{x} \geq 1$ for $\nv$-a.e.\ $x \in \M$, then $C$ is a rectifiable cone. The scaling invariance of $\muc$ will be an easy consequence of this.
We first claim that $C$ is rectifiable. In fact, we have:
\begin{itemize}
\item $\tk(\norm{V_j},x) \geq 1 $ for $\norm{V_j}$-a.e.\ $x \in \M_j$, since $\tkv{x} \geq 1$ for $\nv$-a.e.\ $x \in \M$;
\item $\sup_{j \in \N} \norm{V_j}(B_1) < \infty$ by \eqref{eq:masses_scalings_unif_bounded};
\item the $V_j$'s have locally uniformly bounded first variations because $0 \in \reg(V)$, by \eqref{eq:estimate_curvature_blow-up} and the uniform bound on \eqref{eq:uniform_bound_mu_blow-up}.
\end{itemize}
Thus we can apply to the sequence $V_j$ the compactnes theorem for rectifiable varifolds \cite[Theorem 42.7]{simon:lectures}, which proves that $C$ is rectifiable and that $\tk(\norm{C},x) \geq 1$ for $\norm{C}$-a.e.\ $x \in \M$.
In particular, there exists a $k$-rectifiable set $\Gamma \subset \W$ such that
\begin{equation*}
C = \theta(x) \haus{k}\llcorner \Gamma
\end{equation*}
with $\theta(x) = \T^k(\norm{C},x) \geq 1$ for $\haus{k}$-a.e.\ $x \in \Gamma$.

It remains to show that $C$ is a cone.
The argument is exactly the same as \cite[Theorem 19.3]{simon:lectures}, but we recall it for the reader convenience.
Since $C$ is rectifiable, to prove
\begin{equation}\label{eq:tangent_varifolds_are_cones}
(\dil{0,\lambda})_\sharp C
=
C
\qquad
\forall \lambda >0,
\end{equation}
(that is the fact that $C$ is a cone), it is enough to show that
\begin{equation}\label{eq:homogeneity_density_cone}
\T^k(\norm{C},\lambda x)
=
\T^k(\norm{C},x)
\qquad
\forall x \in \R^n,
\forall \lambda >0,
\end{equation}
that is $\theta$ is homogeneous of degree $0$. This is clearly implied by
\begin{equation}\label{eq:homogeneous_mass_cone}
\norm{C}(\lambda A)
=
\lambda^k \norm{C}(A)
\qquad
\forall A \subset \R^n \mbox{ Borel, }
\forall \lambda >0.
\end{equation}
By approximation, it is enough to prove that, for every $0$-homogeneous function $h \in C^1(\R^n)$, one has
\begin{equation}\label{eq:invariance_scalings_homogeneous}
\frac{\dif}{\dif \lambda}
\left(
\frac{1}{\lambda^k}
\int_{\M} \gamma\Big( \frac{|x|}{\lambda}\Big) h(x)
\dif \nv(x)
\right)
=
0
\qquad
\forall \lambda>0,
\end{equation}
where $\gamma$ is the cut-off function defined in Section \ref{sec:notations}.
To this aim, if $\lambda>0$ is such that $\norm{C}(\partial B_\lambda)=0$, we observe that
\begin{equation*}
\begin{split}
\frac{\norm{C}(B_\lambda)}{\lambda^k}
=
\lim_{j \to \infty}
\frac{\norm{V_j}(B_\lambda)}{\lambda^k}
=
\lim_{j \to \infty} \frac{\nv(B_{\lambda r_j})}{ \lambda^k r_j^k}
=
\T^k(\nv,0).
\end{split}
\end{equation*}
Since there exists at most a countable number of radii $\lambda_j$ such that $\norm{C}(\partial B_{\lambda_j})>0$, by approximation it follows that
\begin{equation}\label{eq:constant_density_cone}
\frac{\norm{C}(B_\lambda)}{\lambda^k}
=
\T^k(\nv,0)
\qquad
\forall \lambda>0.
\end{equation}
If we write the monotonicity identity \eqref{eq:monotonicity_identity_fb} for $C$, since $H_C = \tilde{H_C}=0$ and $\partial( \W)$ is flat, integrating between $\sigma, \lambda$ we obtain
\begin{equation}\label{eq:monotonicity_identity_cone}
\begin{split}
\frac{1}{\lambda^k}
\int_{W} \gamma\Big( \frac{|x|}{\lambda}\Big)
\dif \norm{C}(x)
& - 
\frac{1}{\sigma^k}
\int_{W} \gamma\Big( \frac{|x|}{\sigma}\Big)
\dif \norm{C}(x)
\\
&
\qquad
=
\frac{1}{\lambda^{k}} \int_{G_k(W)} \gamma\Big( \frac{|x|}{\lambda}\Big) \bigl| P_{S^\perp}\nabla |x| \bigr|^2 \dif C(x,S)
\\
&
\qquad -
\frac{1}{\sigma^{k}} \int_{G_k(W)} \gamma\Big( \frac{|x|}{\sigma}\Big) \bigl| P_{S^\perp}\nabla |x| \bigr|^2 \dif C(x,S)
\\
&
\qquad +
 \int_{G_k(W)}\bigl| P_{S^\perp}\nabla |x| \bigr|^2
 \left(
 \int_\sigma^\lambda \frac{k}{\rho^{k+1}} \gamma\Big( \frac{|x|}{\rho}\Big) \dif \rho
 \right)
 \dif C(x,S)
%\\
%= 
%& \frac{1}{\rho^{k+1}} \int_\M \scal{X}{\widehat{H}} \dif \nv
%- \frac{1}{\rho^{k+1}} \int_{\pM} \scal{X}{N} \dif \muv
%\\
%& \qquad
%-
%\frac{1}{\rho^{k+1}} \int_{G_k(\M)} \frac{|x|}{\rho} \gamma'\Big( \frac{|x|}{\rho}\Big) \left| P_{S^\perp}\frac{x}{|x|}\right|^2 \dif V(x,S)
%\\
%\geq &
%- \frac{1}{\rho^{k}} \int_\M \gamma\Big( \frac{|x|}{\rho}\Big)|\widehat{H}| \dif \nv
%- \frac{1}{\rho^{k}} \int_{\pM} \gamma\Big( \frac{|x|}{\rho}\Big)
%\scal{\frac{x}{|x|}}{N}
%\dif \muv
\end{split}
\end{equation}
Letting $\gamma$ increase to $\ind_{[0,1)}$ in \eqref{eq:monotonicity_identity_cone}, by dominated convergence  and \eqref{eq:constant_density_cone} we get
\begin{equation}\label{eq:monotonicity_cone}
\begin{split}
0
=
\frac{\norm{C}(B_\lambda)}{\lambda^k}
-
\frac{\norm{C}(B_\sigma)}{\sigma^k}
=
\int_{G_k(B_\lambda \setminus B_\sigma)}\frac{\bigl| P_{S^\perp}\nabla |x| \bigr|^2}{|x|^k} \dif C(x,S).
\end{split}
\end{equation}
Since the last term is non-negative, we have
\begin{equation}\label{eq:projection_radial_zero}
\bigl| P_{S^\perp}\nabla |x| \bigr|^2
=
\Bigl| P_{S^\perp} \frac{x}{|x|} \Bigr|^2
=
0
\qquad
\mbox{for } C\mbox{-a.e. } (x,S) \in G_k(\W),
\end{equation}
that is
\begin{equation}\label{eq:cone_radial_planes}
C \big( (x,S) \mid P_S(x) \neq x \big)
=
C \big( (x,S) \mid x \notin S \big)
=
0.
\end{equation}
If $h$ is $0$-homogeneous, then $\scal{\nabla h(x)}{x}=0$ and \eqref{eq:cone_radial_planes} implies
\begin{equation}\label{eq:cone_homogeneous_function}
\scal{ \nabla h (x) }{P_S \, x}
=
\scal{ \nabla h (x) }{x}
=0
\qquad
\mbox{for } C\mbox{-a.e. } (x,S) \in G_k(\W).
\end{equation}
Since $C$ has free boundary at $\partial \W$, by testing  \eqref{eq:generalized_mean_curvature} for $C$ with $X(x)=h(x) \gamma\Big( \frac{|x|}{\lambda}\Big) x \in \X_t(\W)$, one obtains
\begin{equation*}
\begin{split}
\frac{\dif}{\dif \lambda}
\left(
\frac{1}{\lambda^k}
\int_{\M} \gamma\Big( \frac{|x|}{\lambda}\Big) h(x)
\dif \nv(x)
\right)
= &
-
\frac{1}{\lambda^{k+1}}
\int_{G_k(\W)} \dive_S X(x) \dif V(x,S)
\\
& +
\frac{1}{\lambda^{k+1}}
\int_{G_k(\W)} \gamma\Big( \frac{|x|}{\lambda}\Big)
\scal{ \nabla h (x) }{P_S \, x}
\dif V(x,S)
\\
= &
\,
0,
\end{split}
\end{equation*}
where the last equality follows by $H_C =0$ and \eqref{eq:cone_homogeneous_function}. This proves \eqref{eq:invariance_scalings_homogeneous}, thus $C$ is a cone. The scaling invariance of $\muc$ is a trivial consequence of \eqref{eq:total_fvf_varifold_free_boundary} applied to $C$ and \eqref{eq:tangent_varifolds_are_cones}.
\end{steps}
\end{proof}

Before of going on, we highlight
%that for $\mus$-a.e.\ $x \in \M$, every $(k-1)$-tangent measure of $\mus$ is $\muc$ for some tangent cone $C$ to $V$. This is
an easy consequence of Lemma \ref{lmm:mus_abscont_hausk-1}, Lemma \ref{lmm:tangent_mus_equal_muv} and Lemma \ref{lmm:tangent_cones}.
%and in particular of \eqref{eq:scaled_mu_is_muj} and \eqref{eq:scalings_muv_to_muc}
\begin{corollary}\label{cor:tangent_mus_are_muc}
Let $V \in \V_k(\M)$ has free boundary at $\pM$ with $H \in L^p(\M, \nv)$ for some $p > 1$ and let $\muv$ be the measure provided by Corollary \ref{cor:bv_free_boundary_varifolds}. Then
\begin{equation}\label{eq:tangent_mus_are_boundary2}
\tang^{k-1}(\mus,x_0)
=
\{
\muc \mid C \in \tang(V,x_0)
\}
\qquad
\mbox{for } \mus\mbox{-a.e. } x_0 \in \pM.
\end{equation}
\end{corollary}
\begin{proof}
By Lemma \ref{lmm:tangent_cones} it clearly follows that for each $x_0 \in \reg(V) \cap \pM$,
\begin{equation*}
\tang^{k-1}(\muv,x_0)
=
\{
\muc \mid C \in \tang(V,x_0)
\}.
\end{equation*}
By Lemma \ref{lmm:mus_abscont_hausk-1} and Lemma \ref{lmm:tangent_mus_equal_muv} we have the conclusion.
\end{proof}

It is well-known that, for a cone, points with maximal density form a linear subspace and that the cone is invariant by translation with respect these points (see e.g. \cite[Section 3.3]{simon1996theorems} and \cite[Theorem 3.1, Example (4) of Section 4]{White1997}). We report here the simple proof of this fact for the sake of completeness.

\begin{lemma}\label{lmm:max_density_invariance}
%Let $V \in \V_k(\M)$ a rectifiable varifold with free boundary at $\pM$, \mbox{$H \in L^p(\M, \nv)$} for some $p\geq 1$ and $\tkv{x} \geq 1$ for $\nv$-a.e.\ $x \in \M$.
Let $V$ satisfy Assumption \ref{ass:varifold_rectifiability}, let $x_0 \in \reg(V) \cap \pM$ be fixed and let $C$ be a tangent cone to $V$ at $x_0$. Then the set
\begin{equation*}
D_C
=
\{y \in T_{x_0} \pM \mid \T^k(\norm{C},y) = \T^k(\norm{C},0)\}
\end{equation*}
is a linear subspace of $\R^{n}$. Moreover the translated cone $(\trasl_{y})_\sharp C$  coincides with $C$ for all $y \in D_C$. In particular, if $\muc$ is the measure given by Corollary \ref{cor:bv_free_boundary_varifolds} relative to $C$, then $(\trasl_{y})_\# \muc = \muc$.

\end{lemma}
\begin{defi}
We say that $D_C$ is the \emph{invariant subspace} of the cone $C$.
\end{defi}

\begin{proof}[Proof of Lemma \ref{lmm:max_density_invariance}]
Without loss of generality we can assume that $x_0=0$. Let us call $\theta_0= \T^k(\nv,0)$. By Lemma \ref{lmm:tangent_cones} $C$ is rectifiable and it holds $\T^k(\nv,0) = \T^k(\norm{C},0)$. Since $C$ is a cone, for $y \in T_x \pM$ we have
\begin{equation}\label{eq:max_density_cone}
\theta_0
\stackrel{\eqref{eq:constant_density_cone}}{=}
\lim_{r \to +\infty}  \frac{\norm{C}\big(B_r\big)}{r^k}
\geq
\lim_{r \to +\infty} \frac{\norm{C}\big(B_{r-|y|}(y)\big)}{(r- |y|)^k}
\frac{(r- |y|)^k}{r^k}
=
\lim_{r \to +\infty} \frac{\norm{C}\big(B_r(y)\big)}{r^k}
\geq
\T^k(\norm{C},y).
%=
%\theta_0.
\end{equation}
The last inequality is given by the monotonicity identity for $C$: in the last member of \eqref{eq:monotonicity_identity_fb}, the first two integrals disapper (since $H+\tilde{H}=0$ and $\scal{N(x)}{x}=0$) and the last term is non-negative.
\eqref{eq:max_density_cone} shows that
\begin{equation}\label{eq:cone_density_max_origin}
\T^k(\norm{C},0)
\geq
\T^k(\norm{C},y)
\qquad
\forall y \in T_0 \pM.
\end{equation}
If $y \in D_C$, then \eqref{eq:max_density_cone} yields
\begin{equation*}
\frac{\norm{C}\big(B_r(y)\big)}{r^k}
=
\theta_0
\qquad
\forall r >0.
\end{equation*}
By the same arguments of Lemma \ref{lmm:tangent_cones}, it follows that $C$ is a rectifiable cone also with respect to $y$.
By rectifiability of $C$, in order to show that $(\trasl_{y})_\sharp C=C$, it is enough to prove that
\begin{equation}\label{eq:points_D_C_invariance_density}
\T^k(C,z)=\T^k(C,y+z)
\qquad
\forall z \in T_0^+ \M.
\end{equation}
To this aim, let $z \in T_0^+\M$ be an arbitrary point.
Since $C$ is a cone with respect to $y$, we have that
\begin{equation*}
\T^k(C,z)
=
\T^k\Big(C, y + \frac{1}{2}(z-y)\Big)
=
\T^k\Big(C, \frac{1}{2}(y+z)\Big).
\end{equation*}
On the other hand, since $C$ is a cone we have
\begin{equation*}
\T^k(\norm{C},y+z) = \T^k\Big(C, \frac{1}{2}(y+z)\Big).
\end{equation*}
This shows \eqref{eq:points_D_C_invariance_density} and hence $(\trasl_{y})_\sharp C=C$.
The translation invariance of $\muc$ is a trivial consequence of \eqref{eq:total_fvf_varifold_free_boundary} applied to $C$ and of $(\trasl_{y})_\sharp C=C$.

It remains to show that $D_C$ is a linear subspace of $\R^{n}$. Since $C$ is a cone, if $y \in D_C$, then $\lambda y \in D_C$ for each $\lambda>0$. Since $C$ is a cone also with respect to $y$, then $\lambda y \in D_C$ also if $\lambda<0$. If $y,z \in D_C$, it follows from the previous discussion that also $y+z \in D_C$ and this proves that $D_C$ is a linear subspace.
\end{proof}

Before going on, we first recall the definition of \emph{approximate continuity}:
\begin{defi}[Approximate continuity]
If $\mu$ is a positive Radon measure and $f \colon \R^n \to \R$ is a Borel function, we say that $f$ is approximate continuous at $x \in \R^n$ with respect to $\mu$ if
\begin{equation*}
\lim_{r \to 0}
\frac{\mu
\big( \left\{
z \in B_r(x) \mid
\left|
f(z)
-
f(x)
\right|
> \e
\right\} \big)}
{\mu ( B_r(x) )}
=
0
\qquad
\forall \e>0.
\end{equation*}
\end{defi}
\begin{rem}\label{rem:measurable_approx_continuous}
It is well-known that, if $\mu$ is a Radon measure, then every $\mu$-measurable function is approximate continuous at $\mu$-a.e.\ point (see e.g. \cite[Theorem 1.37]{evans2015measure} where the proof is done for the Lebesgue measure, but the same arguments can be applied to any Radon measure).
\end{rem} 
The following lemma states that there exists a set $F$ of full $\muv$-measure with respect to $\reg(V)$ such that for every $x \in F$ the invariant subspace $D_C$ of any cone $C \in \tang(V,x)$ coincides with $\supp \muc$.

\begin{lemma}\label{lmm:approx_continuity_contant_density}
%Let $V \in \V_k(\M)$ a rectifiable varifold with free boundary at $\pM$ with \mbox{$H \in L^p(\M, \nv)$} for some $p \geq 1$ and $\tkv{x} \geq 1$ for $\nv$-a.e.\ $x \in \M$.
Let $V$ satisfy Assumption \ref{ass:varifold_rectifiability}.
Then there exists a set $F \subset \reg(V) \cap \pM$ that satisfies
\begin{equation*}
\muv(\reg(V) \- F) = 0
\end{equation*}
and has the following property: for every $x_0 \in F$ and for every
\mbox{$C \in \tang(V,x_0)$} we have either $\muc=0$ or $\muc \neq 0$ and
\begin{equation*}
\T^k(\norm{C},y)
=
\T^k(\norm{C},0)
\qquad
%\mbox{ for }
%\mu_{C} \mbox{-a.e. } y \in T_{x_0} \pM.
\forall y \in \supp \muc.
\end{equation*}
In particular, either $\muc=0$ or $\supp \muc = D_C$.
\end{lemma}

\begin{proof}
The idea of the proof is the following: we first define the ``good" set $F$ of full $\muv$-measure with respect to $\reg(V)$ where $\tkv{\cdot}$ exists and is approximate continuous with respect to $\muv$.
Next, we fix $x_0 \in F$, $r_j \to 0$ and, using the notations for the scalings, 
\begin{equation*}
V_j \wto C \in \tang(V,x_0).
\end{equation*}
We fix $y \in \supp \muc$ and we assume by contradiction that the statement is false. 
We find a tiny ball $B_r(y)$ where $\T^k(\norm{V_j},\cdot)$ is close to $\Tn(\norm{C},y)$ for $j$ sufficiently large. Since $\Tn(\norm{C},y)$ is close to $\norm{C}(B_\rho(y))/\rho^k$ for small $\rho$, this is achieved by weak convergence $V_j \wtoi C$ and using the properties of $\varphi_{x_0}$ stated in Theorem \ref{thm:singular_set_V_lower_dimensional}.
Since $\muj(B_r(y)) > \beta >0$ for large $j$, this contradicts the approximate continuity of $\Tn(\nv, \cdot)$ at $x_0$ with respect to $\muv$.
\begin{steps}[wide,%
%labelwidth=!,%
labelindent=5pt]
\item We are going first to define the set $F$ of full $\muv$-measure with respect to $\reg(V)$ and next we will prove that the conclusion of the theorem holds for every $x \in F$.

We call $A$ the set of points $x \in \pM$ that satisfy all the following conditions:
\begin{enumerate}
\item $x \in \reg(V) \cap \supp \muv$ and is of density $1$ for $\reg(V)$ with respect to $\muv$, that is
\begin{equation}\label{eq:density_1_regV}
\lim_{\rho \to 0} \frac{\muv \big( \reg(V) \cap B_\rho(x)\big)}{\muv\big(B_\rho(x) \big)}=1;
\end{equation}
\item $x$ is a point of approximate continuity for $\tkv{\cdot}$  with respect to $\muv$ (which is a well-defined Borel function in $\reg(V)$ and can be extended, for instance, to be $0$ outside $\reg(V)$; the chosen extension does not influence the approximate continuity, because by \eqref{eq:density_1_regV} $x$ has density $1$ in $\reg(V)$, where $\tkv{\cdot}$ is well-defined).
\end{enumerate}
By Lebesgue differentiation theorem and by Remark \ref{rem:measurable_approx_continuous}, we have
\begin{equation}\label{eq:full_mus_measure_A}
\muv (\reg(V) \- A)
= 0.
\end{equation}
%, $\muv$-a.e.\ $x \in \reg(V)$ is of approximate continuity for $\tkv{\cdot}$ with respect to $\muv$.
In addition, by Theorem \ref{thm:singular_set_V_lower_dimensional}, for every $x \in A$ the map $\rho \mapsto \varphi_x ( \rho )$ is monotone increasing and converge pointwise to $0$ as $\rho \downarrow 0$, that is
\begin{equation*}
\lim_{\rho \to 0} \varphi_x ( \rho )
= 0
\qquad
\forall x \in A.
\end{equation*}
Thus, by Egoroff's Theorem, for every $h \in \N$ there exists a set $F_h \subset A$ such that
\begin{equation}\label{eq:uniform_convergence_fh}
\muv(A \- F_h) \leq 1/h,
\qquad
\varphi_x (\rho)
\utoi_{\rho \to 0}
0
\mbox{ on }
F_h.
\end{equation}
Up to removing sets of $\muv$-measure $0$, we can assume that every $x \in F_h$ is a point of density $1$ with respect to $\muv$, that is
\begin{equation}\label{eq:Fh_density_1}
\lim_{\rho \to 0} \frac{\muv\big( F_h \cap B_\rho(x)\big)}{\muv\big(B_\rho(x)\big)}
= 1
\qquad
\forall x \in F_h.
\end{equation}
We now define
\begin{equation*}
F = A \cap
\Big(
\bigcup_{h \in \N} F_h
\Big).
\end{equation*}
By \eqref{eq:full_mus_measure_A} and \eqref{eq:uniform_convergence_fh} it follows
\begin{equation*}
\mus( \reg(V) \- F) = 0.
\end{equation*}

Let us fix $x_0 \in F$ and let us consider $h \in \N$ such that $x_0 \in F_h$. Without loss of generality we can assume $x_0=0$.

Now let us fix $r_j \downarrow 0$.
Using the notations for the scalings, since $F \subset \reg(V)$, by Lemma \ref{lmm:tangent_cones} there exists a subsequence, not relabeled, such that $V_j \wtoi C$ with $C \in \tang(V,0)$ and $C$ is a rectifiable cone with $\tk(\norm{C},y)\geq 1$ for $\norm{C}$-a.e.\ $y \in \R^n$.

We now need a technical remark that is useful in the rest of the proof:
recalling that $H_j(y)=r_jH(r_jy)$, by a simple change of variables one obtains
\begin{equation*}
r_j x \in \reg(V)
\quad
\Leftrightarrow
\quad
x \in \reg(V_j).
\end{equation*}
More precisely, if we call $\varphi_x^j$ the function for $V_j$ defined in Theorem \ref{thm:singular_set_V_lower_dimensional}, we get
\begin{equation}\label{eq:scaling_varphi}
\varphi_x^j(\rho)
=
\varphi_{r_j x} (r_j \rho).
\end{equation}
\item
We can now begin with the proof. If $\muc=0$ there is nothing to prove. Thus we can assume $\muc \neq 0$ and $\emptyset \neq \supp \muc \subset \partial \W$. Since $C$ is a cone,
By \eqref{eq:cone_density_max_origin} we have
\begin{equation*}
\Tn(\norm{C},y)  \leq \Tn(\norm{C},0)
\qquad
\forall y \in \partial \W.
\end{equation*}
By contradiction, let us assume that there exists $y \in \supp \muc$ and $\e >0$ such that
\begin{equation}\label{eq:density_cone_contradiction}
\Tn(\norm{C},y)
< 
\Tn(\norm{C},0) - \e.
\end{equation}
Since $\muc$ is scaling invariant by Lemma \ref{lmm:tangent_cones}, without loss of generality we can assume that $y\in B_{1/2}$.
By the uniform convergence \eqref{eq:uniform_convergence_fh} and by definition of density, there exists $\rho \in
%(0,\min\{\frac{1}{2}, \frac{\kappa(\M)}{2}\})
(0,1/2)
$
such that
\begin{equation}\label{eq:density_cone_close_limit}
\frac{\norm{C}\big(B_\rho(y)\big)}{\rho^k}
+ \varphi_z(\rho)
\leq
\Tn(\norm{C},y)
+ \frac{\e}{8}
\qquad
\forall z \in F_h.
\end{equation}
Since $V_j \wtoi C$, without loss of generality we can choose $\rho$ so that there exists $J \in \N$ and a small $0< r < \rho$ for which
\begin{equation}\label{eq:closeness_density_cones_scalings}
\left|
\frac{\norm{C}\big(B_\rho(y)\big)}{\rho^k}
-
\frac{\norm{V_{j}}\big(B_{\rho}(y)\big)}{(\rho-r)^k}
\right|
< \frac{\e}{8}
\qquad
\forall j > J.
\end{equation}
Let us choose $j> J$; for every $z \in B_r(y)$ such that $r_j z \in F_h$ we have
\begin{equation*}
\begin{split}
\Tn(\norm{V_{j}},z)
%& \stackrel{\mbox{\small Cor.} \ref{cor:monotonicity_formula}}{\leq}
& \leq
\frac{\norm{V_{j}}\big(B_{\rho-r}(z)\big)}{(\rho-r)^k}
+ \varphi_z^j(\rho - r)
\\
%& \stackrel{B_{\rho -r}(z) \subset B_\rho(y)}{\leq}
& \leq
\frac{\norm{V_{j}}\big(B_{\rho}(y)\big)}{(\rho-r)^k}
+ \varphi_{r_j z}(r_j \rho)
\\
& \stackrel{\eqref{eq:closeness_density_cones_scalings}}{\leq}
\frac{\norm{C}\big(B_\rho(y)\big)}{\rho^k}
+ \varphi_{r_j z}(r_j \rho)
+ \frac{\e}{8}
\\
& \stackrel{\eqref{eq:density_cone_close_limit}}{\leq}
\Tn(\norm{C},y) + \frac{\e}{4}
\\
& \stackrel{\eqref{eq:density_cone_contradiction}}{\leq}
\Tn(\norm{C},0) - \frac{3 \e}{4}
\\
& =
\Tn(\norm{V_j},0) - \frac{3 \e}{4}.
\end{split}
\end{equation*}
where we used the fact that every $\varphi_x$ is increasing and \eqref{eq:scaling_varphi}.
This shows that, for $j > J$,
\begin{equation*}
B_r(y) \cap \frac{1}{r_j}F_h
\subseteq
\left\{
z \in B_1 \mid
\left|
\Tn(\norm{V_j},z)
-
\Tn(\norm{V_j},0)
\right|
> \frac{\e}{2}
\right\}.
\end{equation*}
\item
We now want to estimate from below the measure of this set to get a contradiction with the approximate continuity of $\tkv{\cdot}$ in $0$.

By approximate continuity of the $\tkv{\cdot}$ in $0$ with respect to $\muv$ we have
\begin{equation}\label{eq:first_estimate_approx_cont}
\begin{split}
0
& =
\limsup_{j \to \infty}
\frac{\muv
\left( \left\{
z \in B_{r_j} \mid
\left|
\T^k(\norm{V},z)
-
\T^k(\norm{V},0)
\right|
> \frac{\e}{2}
\right\} \right)}
{\muv ( B_{r_j} )}
\\
& \geq
\limsup_{j \to \infty} \frac{\muv \big( B_{rr_j }(r_j y) \cap F_h \big)}{\muv(B_{r_j})}
\\
& =
\limsup_{j \to \infty}
\frac{\muv \big( B_{rr_j }(r_j y)  \big)}{\muv(B_{r_j})}
\\
& =
\limsup_{j \to \infty}
\frac{\muj \big( B_{r }(y)  \big)}{\muj(B_{1})},
\end{split}
\end{equation}
where the second identity is consequence of \eqref{eq:Fh_density_1} and the last one follows by \eqref{eq:scaled_mu_is_muj}.
We want to estimate from below the last term to get a contradiction. To do so, let us notice that, by $\muj \wto \muc$ and $y \in \supp \muc$, there exist two constants $c,\beta >0$ such that, for $j$ sufficiently large,
\begin{equation*}
\muj \big( B_{1} \big)
\leq c,
\qquad
\muj \big( B_r(y)\big)
\geq \beta,
\end{equation*}
which contradicts \eqref{eq:first_estimate_approx_cont}.

This also prove the inclusion $\supp \muc \subset D_C$. To prove the other inclusion, let us notice that the scaling invariance of $\muc$ and $\supp \muc \neq \emptyset$ imply $0 \in \supp \muc$. Since $\muc$ is invariant by translations along $D_C$ by Lemma \ref{lmm:max_density_invariance}, we have the opposite inclusion and $\supp \muc = D_C$.
\end{steps}
\end{proof}

We next prove that, for every $x \in F$ (where $F$ is the set defined in the previous Lemma) such that $\tlu(\muv,x)>0$ and for every $C \in \tang(V,x)$, $\muc$ is the surface measure of a $(k-1)$-plane, which coincides with $D_C$.

\begin{lemma}\label{lmm:positive_density_k-1-dimensional}
%Let $V \in \V_k(\M)$ be a rectifiable varifold with free boundary at $\pM$ with generalized mean curvature \mbox{$H \in L^p(\M, \nv)$} for some $p \geq 1$, $\tkv{x} \geq 1$ for $\nv$-a.e.\ $x \in \M$
Let $V$ satisfy Assumption \ref{ass:varifold_rectifiability}
and let $F$ be the set defined in Lemma \ref{lmm:approx_continuity_contant_density}. For every $x_0 \in F$ such that $\tlu(\muv,x)>0$ and for every $C \in \tang(V,x_0)$, the invariant subspace $D_C$ of $C$ is $(k-1)$-dimensional; moreover there exists $\alpha_0 > 0$ such that
\begin{equation*}
\muc= \alpha_0 \haus{k-1} \llcorner D_C
%\qquad
%\alpha_0
%\geq
%\liminf_{r \to 0}\frac{\muv(B_r(x))}{\omega_{k-1} r^{k-1}}
.
\end{equation*}
%In particular, for $\mus$-a.e. $x_0 \in \pM$ and for every $C \in \tang(V,x_0)$, $\muc$ is a $(k-1)$-dimensional plane.
\end{lemma}

\begin{proof}
Let $x_0 \in F$ be a fixed point and let us assume $C \in \tang(V,x_0)$.
Without loss of generality we can assume that $x_0=0$ and that $T_0 \pM$ is the subspace $\{x_n = 0\}$. Since $C \in \tang(V,0)$ there exists $r_j \downarrow 0$ such that, using the notations for the scalings,
%since $F \subset \reg(V)$, by Lemma \ref{lmm:tangent_cones} there exists a subsequence, not relabeled, such that
$V_j \wtoi C$. Lemma \ref{lmm:tangent_cones} asserts $C$ is a rectifiable cone and that that $\muj \wto \muc$, where $\muc$ is the measure relative to $C$ given by Corollary \ref{cor:bv_free_boundary_varifolds}.

We first recall that the condition
\begin{equation}\label{eq:positive_k-1_density}
\tlu(\muv,0)
=
\liminf_{r \to 0}\frac{\muv\big( B_r(x_0) \big)}{r^{k-1}}
>0,
\end{equation}
together with $\muj \wto \muc$, implies that $\muc \neq 0$. Thus, Lemma \ref{lmm:approx_continuity_contant_density} provides $\supp \muc = D_C$, where $D_C$ is the invariant subspace $D_C$ of $C$, given by Lemma \ref{lmm:max_density_invariance}.
$D_C$ is a linear subspace of $\R^n$ and throughout this proof we call $m= \dim D_C$ its dimension. By definition of $D_C$, we clearly have $D_C \subset \{x_n = 0\}$. Thus $m \leq n-1$. After a suitable change of coordinates, we can assume that $D_C= \{x_{m+1}= \dots = x_{n}=0\}$.

For any $y \in D_C$ and any $r>0$, we denote by $Q_{D_C}(y,r)$ the closed cube included in $D_C$ with center $y$, side of length $r$ and faces parallel to the coordinate vectors $e_1, \dots, e_{m}$.
If we set
\begin{equation}\label{eq:defi_alpha_0}
\alpha_0
=
\liminf_{j \to \infty}\frac{\muv\big(B_{r_j}\big)}{\omega_{k-1} r_j^{k-1}}.
%=
%\liminf_{j \to \infty}\frac{\muj \big(B_{1}\big)}{\omega_{k-1}},
\end{equation}
we have $\alpha_0>0$, by \eqref{eq:positive_k-1_density}. 
Since $\muc$ is invariant by scalings (by Lemma \ref{lmm:tangent_cones}) and by translations in $D_C$ (by Lemma \ref{lmm:max_density_invariance}), there exists a fixed $\beta_0>0$ such that
\begin{equation}\label{eq:invariance_measure_cubes}
\muc \big( Q_{D_C}(y,r) \big) = \beta_0 r^{k-1}
\qquad
\forall y \in D_C
\quad
\forall r >0.
\end{equation}
We are going to show that \eqref{eq:invariance_measure_cubes} implies $m = k-1$. We argue by contradiction and by cases:
\begin{itemize}
\item Let us assume, by contradiction, that $ m < k-1$. For each $l \in \N$, there exists a covering $\{Q_i^l\}_{i=1}^{2^{lm}}$ of $Q_{D_C}(0,1)$ such that each $Q_i^l$ is a cube included in $D_C$ and of side length $2^{-l}$.
Therefore
\begin{equation*}
\muc\big( Q_{D_C}(0,1) \big)
\leq
\sum_{i=1}^{2^{lm}} \muc \big( Q_i^l \big)
=
2^{lm} \beta_0 2^{-l(k-1)}
\leq
\beta_0 2^{-l}
\xrightarrow[l \to \infty]{}
0.
\end{equation*}
Thus $\muc\big( Q_{D_C}(0,1) \big)=0$. By translation invariance of $\muc$, it follows that $\muc = 0$, which is a contradiction.
\item Let us assume now that $m \geq k$. We first observe that by approximation, \eqref{eq:invariance_measure_cubes} holds also for cubes that are open in $D_C$. Hence, taking for every $l \in \N$ a covering $\{Q_i^l\}_{i=1}^{2^{lm}}$ of $Q_{D_C}(0,1)$ of cubes included in $D_C $ with disjoint interiors and of side length $2^{-l}$, we have
\begin{equation*}
\muc \big( Q_{D_C}(0,1) \big)
\geq
\sum_{i=1}^{2^{lm}} \muc \Big(  \interior{(Q_i^l)} \Big)
=
2^{lm} \beta_0 2^{-l(k-1)}
\geq
\beta_0 2^{l}
\xrightarrow[l \to \infty]{}
+\infty,
\end{equation*}
(where $\interior{(Q_i^l)}$ is intended in the topology of $D_C$) 
which is a contradiction.
\end{itemize}
This shows that $\dim D_C = k-1$. Since $\muc$ is invariant by scalings and translations in $D_C$, we have
\begin{equation*}
\frac{\muc \big( B(r(y) \big)}{\haus{k-1}\big( D_C \cap B_r(y) \big)}
=
\alpha_0
\qquad
\forall y \in D_C,
\,
\forall r>0.
\end{equation*}
%(This shows, in particular, that the $\liminf$ in \eqref{eq:defi_alpha_0} is in fact a limit.)
By Radon-Nikodym Theorem \cite[Theorem 4.7]{simon:lectures}, we obtain that
$\muc = \alpha_0 \haus{k-1} \llcorner D_C$.
%
%The last assertion follows by
%$\mus(\pM \- F) = 0$ as stated in Lemma \ref{lmm:approx_continuity_contant_density}.
\end{proof}

\subsection{Proof of Theorem \ref{thm:rectifiability_free_boundary}}\label{subsec:final}

We can now prove Theorem \ref{thm:rectifiability_free_boundary}.

\begin{proof}[Proof of Theorem \ref{thm:rectifiability_free_boundary}]
By definition of $\mus$ we have
\begin{equation}\label{eq:densities_positive_finite}
0 <
\tlu(\mus,x)
\leq
\tuu(\mus,x) < + \infty
\qquad
\mbox{for } \mus \mbox{-a.e. } x \in \pM,
\end{equation}
thus $\mus$ satisfies the first condition of the Marstrand-Mattila Rectifiability Criterion (Theorem \ref{thm:M-M_criterion}).

To check the second condition of the criterion, let us notice that
Corollary
\ref{cor:tangent_mus_are_muc}
yield
\begin{equation}\label{eq:tangent_mus_are_muc_2}
\tang^{k-1}(\mus,x)
=
\{
\muc \mid C \in \tang(V,x)
\}
\qquad
\mbox{for } \mus \mbox{-a.e. } x \in \pM.
%\forall x_0 \in \reg(V) \cap \pM.
\end{equation}
Since $\tlu(\muv,x)>0$ for every $x \in E$ where $E$ is defined in \eqref{eq:definition_mus}, Lemma \ref{lmm:approx_continuity_contant_density} and Lemma \ref{lmm:positive_density_k-1-dimensional} yields
\begin{equation}\label{eq:muc_are_k-1_dimensional_planes}
\{
\muc \mid C \in \tang(V,x)
\}
\subset
\{
\alpha \haus{k-1} \llcorner S
\mid
S \, (k-1)\mbox{-dimensional plane}, \alpha>0
\}
\qquad
\forall x \in F \cap E,
\end{equation}
where $F$ is the set defined in Lemma \ref{lmm:approx_continuity_contant_density}. Moreover
\begin{equation}\label{eq:F_cap_E_is_mus-full}
\begin{split}
\mus \big( \pM \- (F \cap E) \big)
\leq &
\,
\mus\big( \pM \- \reg(V) \big)
+
\mus \big( \reg(V) \- F\big)
+
\mus \big(\pM \- E \big)
\\
= &
\, 0.
\end{split}
\end{equation}
Every set in the right-hand side is $\mus$-negligible because: $\mus \big( \pM \- \reg(V) \big)=0$ by Lemma \ref{lmm:mus_abscont_hausk-1} and $p>1$; $\mus \big( \reg(V) \- F\big)=0$ by Lemma \ref{lmm:approx_continuity_contant_density} and $\mus \ll \muv$; $\mus\big(\pM \- E \big)=0$ by definition of $\mus$.

Summarizing \eqref{eq:tangent_mus_are_muc_2}, \eqref{eq:muc_are_k-1_dimensional_planes} and \eqref{eq:F_cap_E_is_mus-full}, we obtain
\begin{equation*}
\tang^{k-1}(\mus,x)
\subset
\{
\alpha \haus{k-1} \llcorner S
\mid
\alpha>0, \,
S \text{ is a }(k-1)\mbox{-dimensional plane}
\}
\qquad
\mbox{for } \mus \mbox{-a.e. } x \in \pM.
\end{equation*}
Since this is the second condition for the Marstrand-Mattila Rectifiability Criterion, we have that $\mus$ is \mbox{$(k-1)$-rectifiable}.
\end{proof}

%\addcontentsline{toc}{section}{\refname}
\printbibliography

\end{document}